\documentclass[reqno,10pt,centertags]{amsart} 
\usepackage{amsmath,amsthm,amscd,amssymb,latexsym,esint,upref,stmaryrd,
enumerate,color,verbatim,yfonts}
\usepackage{hyperref} 
\newcommand*{\mailto}[1]{\href{mailto:#1}{\nolinkurl{#1}}}
\newcommand{\arxiv}[1]{\href{http://arxiv.org/abs/#1}{arXiv:#1}}

\newcommand{\R}{{\mathbb R}}
\newcommand{\N}{{\mathbb N}}

\newcommand{\bbC}{{\mathbb{C}}}

\newcommand{\bbN}{{\mathbb{N}}}

\newcommand{\bbR}{{\mathbb{R}}}

\newcommand{\cH}{{\mathcal H}}

\newcommand{\cX}{{\mathcal X}}

\newcommand{\beq}{\begin{equation}}
\newcommand{\enq}{\end{equation}}

\renewcommand{\a}{\alpha}
\renewcommand{\b}{\beta}
\newcommand{\g}{\gamma}

\renewcommand{\l}{\lambda}

\DeclareMathOperator{\supp}{supp}

\DeclareMathOperator{\ran}{ran}
\DeclareMathOperator{\dom}{dom}

\renewcommand{\ln}{\text{\rm ln}}

\newcommand{\no}{\notag}
\newcommand{\lb}{\label}
\newcommand{\f}{\frac}

\newcommand{\ol}{\overline}
\newcommand{\bs}{\backslash}

\newcommand{\wti}{\widetilde}
\newcommand{\Oh}{O}
\newcommand{\oh}{o}
\newcommand{\hatt}{\widehat} 
\newcommand{\dott}{\,\cdot\,}
\newcommand{\linspan}{\operatorname{lin.span}}

\renewcommand{\dot}{\overset{\textbf{\Large.}}}

\renewcommand{\dotplus}{\overset{\textbf{\Large.}} +}

\newcommand{\bi}{\bibitem}

\renewcommand{\le}{\leqslant}

\let\geq\geqslant
\let\leq\leqslant

\newcommand{\lam}{\lambda}

\newcommand{\ga}{\gamma}

\newcommand{\de}{\delta}

\newcommand{\Lr}{{L^2((a,b);rdx)}} 

\newcommand{\ACl}{{AC_{loc}((a,b))}}
\newcommand{\Ll}{{L^1_{loc}((a,b);dx)}}

\makeatletter
\def\theequation{\@arabic\c@equation}
\allowdisplaybreaks 
\numberwithin{equation}{section}

\newtheorem{theorem}{Theorem}[section]

\newtheorem{lemma}[theorem]{Lemma}

\newtheorem{definition}[theorem]{Definition}
\newtheorem{hypothesis}[theorem]{Hypothesis}
\newtheorem{example}[theorem]{Example}

\theoremstyle{remark}
\newtheorem{remark}[theorem]{Remark}

\begin{document}

\title[The Krein--von Neumann Extension Revisited]{The Krein--von Neumann Extension Revisited} 

\author[G.\ Fucci]{Guglielmo Fucci}
\address{Department of Mathematics, 
East Carolina University, 331 Austin Building, East Fifth Street,
Greenville, NC 27858-4353, USA}
%\email{\mailto{fuccig@ecu.edu}}
\email{fuccig@ecu.edu}
%\urladdr{\url{http://myweb.ecu.edu/fuccig/}}
\urladdr{http://myweb.ecu.edu/fuccig/}

\author[F.\ Gesztesy]{Fritz Gesztesy}
\address{Department of Mathematics, 
Baylor University, Sid Richardson Bldg., 1410 S.\,4th Street, Waco, TX 76706, USA}
%\email{\mailto{Fritz\_Gesztesy@baylor.edu}}
\email{Fritz$\_$Gesztesy@baylor.edu}
%\urladdr{\url{http://www.baylor.edu/math/index.php?id=935340}}
\urladdr{http://www.baylor.edu/math/index.php?id=935340}

\author[K.\ Kirsten]{Klaus Kirsten}
\address{Department of Mathematics, 
Baylor University, Sid Richardson Bldg., 1410 S.\,4th Street, Waco, TX 76706, USA, and Mathematical
Reviews, American Mathematical Society, 416 4th Street, Ann Arbor, MI 48103, USA}
%\email{\mailto{Klaus\_Kirsten@baylor.edu}}
\email{Klaus$\_$Kirsten@baylor.edu}
%\urladdr{\url{http://www.baylor.edu/math/index.php?id=54012}}
\urladdr{http://www.baylor.edu/math/index.php?id=54012}

\author[L.\ Littlejohn]{Lance L. Littlejohn}
\address{Department of Mathematics, 
Baylor University, Sid Richardson Bldg., 1410 S.\,4th Street, Waco, TX 76706, USA}
%\email{\mailto{Lance\_Littlejohn@baylor.edu}}
\email{Lance$\_$Littlejohn@baylor.edu}
%\urladdr{\url{http://www.baylor.edu/math/index.php?id=53980}}
\urladdr{http://www.baylor.edu/math/index.php?id=53980}

\author[R.\ Nichols]{Roger Nichols}
\address{Department of Mathematics (Dept.~6956), The University of Tennessee at Chattanooga, 
615 McCallie Ave, Chattanooga, TN 37403, USA}
%\email{\mailto{Roger-Nichols@utc.edu}}
\email{Roger-Nichols@utc.edu}
%\urladdr{\url{http://www.utc.edu/faculty/roger-nichols/index.php}}
\urladdr{http://www.utc.edu/faculty/roger-nichols/index.php}

\author[J.\ Stanfill]{Jonathan Stanfill}
\address{Department of Mathematics, 
Baylor University, Sid Richardson Bldg., 1410 S.\,4th Street, Waco, TX 76706, USA}
%\email{\mailto{Jonathan\_Stanfill@baylor.edu}}
\email{Jonathan$\_$Stanfill@baylor.edu}
%\urladdr{\url{http://sites.baylor.edu/jonathan-stanfill/}}
\urladdr{http://sites.baylor.edu/jonathan-stanfill/}

%\dedicatory{}

\date{\today}
\thanks{Applicable Anal. 2021, 25p., DOI: 10.1080/00036811.2021.1938005.} 
%\thanks{Appeared in {\it .}
\@namedef{subjclassname@2020}{\textup{2020} Mathematics Subject Classification}

%%%%%%%%%%%%%%%%%%%%%%%%%%%%%%%
\begin{abstract} 
We revisit the Krein--von Neumann extension in the case where the underlying symmetric operator is strictly positive 
and apply this to derive the explicit form of the Krein--von Neumann extension for singular, general (i.e., three-coefficient) 
Sturm--Liouville operators on arbitrary intervals. In particular, the boundary conditions for the Krein--von Neumann extension of the strictly positive minimal Sturm--Liouville operator are explicitly expressed in terms of generalized boundary values adapted to the (possible) singularity structure of the coefficients near an interval endpoint. 
\end{abstract}
%%%%%%%%%%%%%%%%%%%%%%%%%%%%%%%

\maketitle

%\newpage 

{\scriptsize{\tableofcontents}}
%\normalsize

%%%%%%%%%%%%%%%%%%%%%%%%%%%%%%%
%%%%%%%%%%%%%%%%%%%%%%%%%%%%%%%
\section{Introduction} \lb{s1} 
%%%%%%%%%%%%%%%%%%%%%%%%%%%%%%%
%%%%%%%%%%%%%%%%%%%%%%%%%%%%%%%

While the principal objective of this paper is to derive the explicit form of the Krein--von Neumann extension for 
singular (three-coefficient) Sturm--Liouville operators on arbitrary intervals with strictly positive underlying minimal operator, we briefly pause and first describe the abstract Krein--von Neumann extension of a nonnegative symmetric operator in complex, separable Hilbert space in a nutshell. 

A linear operator $S:\dom(S)\subseteq\cH\to \cH$ in some complex, separable Hilbert space $\cH$ is called 
{\it nonnegative} if 
\begin{equation}\label{Pos-1}
(u,Su)_\cH\geq 0, \quad u\in \dom(S) 
\end{equation}
(in this case $S$ is symmetric). In addition,  $S$ is called {\it strictly positive}, if for some 
$\varepsilon >0$, $(u,Su)_\cH\geq \varepsilon \|u\|_{\cH}^2$, $u\in \dom(S)$. 
Next, we recall the order relation $0 \leq A \leq B$ for two nonnegative self-adjoint operators in $\cH$ in the form 
(see, e.g., \cite[Section\ I.6]{Fa75}, \cite[Theorem\ VI.2.21]{Ka80})  
\begin{equation}
0 \leq A\leq B  \, \text{ if and only if } \, (B + a I_\cH)^{-1} \leq (A + a I_\cH)^{-1} 
\, \text{ for all $a>0$.}     \label{PPa-1} 
\end{equation} 

In the following $0 \leq S$ is a linear, unbounded, densely defined, nonnegative operator in $\cH$, and we assume that $S$ has nonzero deficiency indices.  In particular,
\begin{equation}
{\rm def} (S) = \dim (\ker(S^*-z I_{\cH})) \in \bbN\cup\{\infty\}, 
\quad z\in \bbC\backslash [0,\infty), 
\lb{DEF}
\end{equation}
is well-known to be independent of $z$ (and if $S \geq \varepsilon I_{\cH}$ for some $\varepsilon > 0$ then the independence of ${\rm def} (S)$ of $z$ extends to $z\in \bbC\backslash [\varepsilon,\infty)$). 
Moreover, since $S$ and its closure $\ol{S}$ have the same self-adjoint extensions in $\cH$, we will without loss of generality assume that $S$ is closed in $\cH$.

The following is a fundamental result that cements the extraordinary role played by the Friedrichs and Krein--von Neumann extensions of $S$, to be found in M.\ Krein's celebrated 1947 paper\footnote{See also Theorems\ 2 and 5--7 in the English summary on page 492.} \cite{Kr47} : 
 
%%%%%%%%%%%%%%%%%%%%%%%%%%%%%%%%%%%%%%%%%%%%
\begin{theorem} \lb{t1.1}
Assume that $S$ is a densely defined, closed, nonnegative operator in $\cH$. Then, among all 
nonnegative self-adjoint extensions of $S$, there exist two distinguished ones, $S_K$ and $S_F$, which are, respectively, the smallest and largest
$($in the sense of order between nonnegative self-adjoint operators\,$)$ such extensions. Furthermore, a nonnegative self-adjoint operator $\wti S$ is a self-adjoint extension of $S$ if and only if $\wti S$ satisfies 
\begin{equation}\label{Fr-Sa}
S_K\leq \wti S \leq S_F.
\end{equation}
In particular, \eqref{Fr-Sa} determines $S_K$ and $S_F$ uniquely. \\[1mm] 
In addition,  if $S\geq \varepsilon I_{\cH}$ for some $\varepsilon >0$, one has 
$S_F \geq \varepsilon I_{\cH}$, and 
\begin{align}
\dom (S_F) &= \dom (S) \dotplus (S_F)^{-1} \ker (S^*),     \lb{SF}  \\
\dom (S_K) & = \dom (S) \dotplus \ker (S^*),    \lb{SK}   \\
\dom (S^*) & = \dom (S) \dotplus (S_F)^{-1} \ker (S^*) \dotplus \ker (S^*)  \no \\
& = \dom (S_F) \dotplus \ker (S^*),    \lb{S*} 
\end{align}
in particular, 
\begin{equation} \label{Fr-4Tf}
\ker(S_K)= \ker\big((S_K)^{1/2}\big)= \ker(S^*) = \ran(S)^{\bot}.
\end{equation} 
\end{theorem}
%%%%%%%%%%%%%%%%%%%%%%%%%%%%%%%%%%%%% 

Here the operator inequalities in \eqref{Fr-Sa} are understood in the resolvent sense, 
\begin{equation}
(S_F + a I_{\cH})^{-1} \le \big(\wti S + a I_{\cH}\big)^{-1} \le (S_K + a I_{\cH})^{-1} 
\, \text{ for some (and hence for all\,) $a > 0$}    \lb{Res}
\end{equation}
(an alternative approach employs quadratic forms).

Thus, $S_K$ and $S_F$ are distinguished self-adjoint extensions of $S$, representing, in particular, extremal 
points of all nonnegative self-adjoint extensions $\wti S \geq 0$ of $S$.

We will call the operator $S_K$ the {\it Krein--von Neumann extension}
of $S$. See \cite{Kr47} and also the discussion in \cite{AS80}, \cite{AT05}. It should be
noted that the Krein--von Neumann extension was first considered by von Neumann 
\cite{vN29} in 1929 in the case where $S$ is strictly positive\footnote{His construction appears in the proof of 
Theorem 42 on pages 102--103.}, that is, if 
$S \geq \varepsilon I_{\cH}$ for some $\varepsilon >0$. However, von Neumann did not isolate the extremal property of this extension as described in \eqref{Fr-Sa} and \eqref{Res}. M.\ Krein \cite{Kr47}, \cite{Kr47a} was the first to systematically 
treat the general case $S\geq 0$ and to study all nonnegative self-adjoint extensions of $S$, illustrating the special role of the {\it Friedrichs extension} (i.e., the ``hard'' extension) $S_F$ of $S$ and the Krein--von Neumann (i.e., the ``soft'') extension $S_K$ of $S$ as extremal cases when considering all nonnegative extensions of $S$. For more results on the 
Krein--von Neumann extension of a strictly positive symmetric operator $S \geq \varepsilon I_{\cH}$ we refer to the 
beginning of Section \ref{s3}. 

However, the principal aim of this paper are (three-coefficient) generally singular Sturm--Liouville differential expressions 
of the type 
\begin{equation}
\tau=\f{1}{r(x)}\left[-\f{d}{dx}p(x)\f{d}{dx} + q(x)\right] \, \text{ for a.e.~$x\in(a,b) \subseteq \bbR$,} 
   \lb{1.1}
\end{equation} 
on a general interval $(a,b) \subseteq \bbR$ and their various $L^2((a,b); rdx)$-realizations, with the coefficients 
$p, q, r$ satisfying Hypothesis \ref{h2.1}. In particular, the 
minimal operator $T_{min}$ associated with $\tau$ (cf.\ \eqref{2.5}), assumed in addition to be strictly positive, plays 
the role of $S$ above, and the corresponding maximal operator $T_{max}$ (cf.\ \eqref{2.2}) represents $S^*$. The 
explicit forms of the Friedrichs and Krein extensions of $T_{min}$ are then of the form \eqref{3.23} and \eqref{3.24}, \eqref{3.25}, \eqref{3.27}, \eqref{3.28}, respectively. Moreover, the corresponding boundary conditions are explicitly expressed in terms of generalized boundary values adapted to the (possible) singularity structure of the coefficients near an interval endpoint with the help of principal and nonprincipal solutions of the underlying Sturm--Liouville equation. 

Briefly turning to a sketch of the content of each section, we note that Section \ref{s2} focuses on the basics of 
Sturm--Liouville operators in $L^2((a,b); rdx)$ and the underlying Weyl--Titchmarsh--Kodaira theory, including self-adjoint 
extensions and generalized boundary values (and conditions) in the singular case. Section \ref{s3} then contains the bulk of the new material in this paper. After continuing a discussion of the abstract Krein--von Neumann extension of a symmetric, strictly positive operator $S \geq \varepsilon I_{\cH}$, an elementary characterization of the Krein--von Neumann extension $S_K$ as the unique self-adjoint extension of $S$ containing $\ker(S^*)$ in its domain is derived  in Lemma \ref{l3.2}. This result is then applied to derive an explicit description of the Krein--von Neumann extension of a strictly positive minimal Sturm--Liouville operator $T_{min}$ in terms of generalized boundary values. We conclude this paper with three nontrivial and representative examples in Section \ref{s4}, including a generalized Bessel operator, a singular operator relevant in the context of acoustic black holes, and the Jacobi operator. 

Finally, a few remarks on the notation employed: Given a separable complex 
Hilbert space $\cH$, $(\dott,\dott)_{\cH}$ denotes the scalar product in $\cH$ (linear in
the second factor), and $I_{\cH}$ represents the identity operator in $\cH$. 
The domain and range of a linear operator $T$ in $\cH$ are abbreviated by 
 $\dom(T)$ and $\ran(T)$. The closure of a closable operator $S$ is
denoted by $\ol S$. The kernel (null space) of $T$ is denoted by
$\ker(T)$. The spectrum, point spectrum (i.e., the set of eigenvalues), and resolvent set of a closed linear operator 
in $\cH$ will be abbreviated by $\sigma(\cdot)$, $\sigma_{p}(\cdot)$, 
and $\rho(\cdot)$, respectively. If $U_1$ and $U_2$ are subspaces of a Banach space $\cX$, their direct sum 
is denoted by $U_1 \dotplus U_2$. We also employ the shortcut $\bbN_0 = \bbN \cup \{0\}$. If the underlying 
$L^2$-space is understood, we denote the corresponding identity operator simply by $I$.

%%%%%%%%%%%%%%%%%%%%%%%%%%%%%%
%%%%%%%%%%%%%%%%%%%%%%%%%%%%%%
\section{The Basics of Weyl--Titchmarsh--Kodaira Theory} \lb{s2}
%%%%%%%%%%%%%%%%%%%%%%%%%%%%%%
%%%%%%%%%%%%%%%%%%%%%%%%%%%%%%

In this section, following \cite{GLN20} and \cite[Ch.~13]{GZ21}, we summarize the singular 
Weyl--Titchmarsh--Kodaira theory as needed to treat the Krein--von Neumann extension for singular, general 
Sturm--Liouville operators in the remainder of this paper. 

Throughout this section we make the following assumptions:

%%%%%%%%%%%%%
\begin{hypothesis} \lb{h2.1}
Let $(a,b) \subseteq \bbR$ and suppose that $p,q,r$ are $($Lebesgue\,$)$ measurable functions on $(a,b)$ 
such that the following items $(i)$--$(iii)$ hold: \\[1mm] 
$(i)$ \hspace*{1.1mm} $r>0$ a.e.~on $(a,b)$, $r\in\Ll$. \\[1mm] 
$(ii)$ \hspace*{.1mm} $p>0$ a.e.~on $(a,b)$, $1/p \in\Ll$. \\[1mm] 
$(iii)$ $q$ is real-valued a.e.~on $(a,b)$, $q\in\Ll$. 
\end{hypothesis}
%%%%%%%%%%%%%

Given Hypothesis \ref{h2.1}, we study Sturm--Liouville operators associated with the general, 
three-coefficient differential expression
\begin{equation}
\tau=\f{1}{r(x)}\left[-\f{d}{dx}p(x)\f{d}{dx} + q(x)\right] \, \text{ for a.e.~$x\in(a,b) \subseteq \bbR$,} 
   \lb{2.1}
\end{equation} 
and introduce maximal and minimal operators in $\Lr$ associated with $\tau$ in the usual manner as follows. 

%%%%%%%%%%%%%%
\begin{definition} \lb{d2.2}
Assume Hypothesis \ref{h2.1}. Given $\tau$ as in \eqref{2.1}, the {\it maximal operator} $T_{max}$ in $\Lr$ associated 
with $\tau$ is defined by
\begin{align}
&T_{max} f = \tau f,    \no
\\
& f \in \dom(T_{max})=\big\{g\in\Lr \, \big| \,g,g^{[1]}\in\ACl;   \lb{2.2} \\ 
& \hspace*{6.35cm}  \tau g\in\Lr\big\}.   \no
\end{align}
The {\it preminimal operator} $\dot T_{min} $ in $\Lr$ associated with $\tau$ is defined by 
\begin{align}
&\dot T_{min}  f = \tau f,   \no
\\
&f \in \dom \big(\dot T_{min}\big)=\big\{g\in\Lr \, \big| \, g,g^{[1]}\in\ACl;   \lb{2.3}
\\
&\hspace*{3.15cm} \supp \, (g)\subset(a,b) \text{ is compact; } \tau g\in\Lr\big\}.   \no
\end{align}

One can prove that $\dot T_{min} $ is closable, and one then defines the {\it minimal operator} $T_{min}$ as the closure of $\dot T_{min} $.
\end{definition}
%%%%%%%%%%%%%%

The following facts then are well known:
\begin{equation} 
\big(\dot T_{min}\big)^* = T_{max}, 
\end{equation} 
and hence $T_{max}$ is closed and $T_{min}=\ol{\dot T_{min} }$ is given by
\begin{align}
&T_{min} f = \tau f, \no
\\
&f \in \dom(T_{min})=\big\{g\in\Lr  \, \big| \,  g,g^{[1]}\in\ACl;     \lb{2.5} \\
& \qquad \text{for all } h\in\dom(T_{max}), \, W(h,g)(a)=0=W(h,g)(b); \, \tau g\in\Lr\big\}   
\no \\
& \quad =\big\{g\in\dom(T_{max})  \, \big| \, W(h,g)(a)=0=W(h,g)(b) \, 
\text{for all } h\in\dom(T_{max}) \big\}.   \no 
\end{align}
Moreover, $\dot T_{min} $ is essentially self-adjoint if and only if\; $T_{max}$ is symmetric, and then 
$\ol{\dot T_{min} }=T_{min}=T_{max}$.

Here the Wronskian of $f$ and $g$, for $f,g\in\ACl$, is defined by
\begin{equation}
W(f,g)(x) = f(x)g^{[1]}(x) - f^{[1]}(x)g(x), \quad x \in (a,b),    \lb{23.2.3.1}  
\end{equation}
with 
\begin{equation}
y^{[1]}(x) = p(x) y'(x), \quad x \in (a,b),
\end{equation}
denoting the first quasi-derivative of a function $y\in AC_{loc}((a,b))$.

The celebrated Weyl alternative then can be stated as follows:

%%%%%%%%%%%%%
\begin{theorem}[Weyl's Alternative] \lb{t2.3} ${}$ \\
Assume Hypothesis \ref{h2.1}. Then the following alternative holds: Either, \\[1mm] 
$(i)$ for every $z\in\bbC$, all solutions $u$ of $(\tau-z)u=0$ are in $\Lr$ near $b$ 
$($resp., near $a$$)$, \\[1mm] 
or, \\[1mm]
$(ii)$  for every $z\in\bbC$, there exists at least one solution $u$ of $(\tau-z)u=0$ which is not in $\Lr$ near $b$ $($resp., near $a$$)$. In this case, for each $z\in\bbC\bs\bbR$, there exists precisely one solution $u_b$ $($resp., $u_a$$)$ of $(\tau-z)u=0$ $($up to constant multiples$)$ which lies in $\Lr$ near $b$ $($resp., near $a$$)$. 
\end{theorem}
%%%%%%%%%%

This yields the limit circle/limit point classification of $\tau$ at an interval endpoint and links self-adjointness of $T_{min}$ (resp., $T_{max}$) and the limit point property of $\tau$ at both endpoints as follows. 

%%%%%%%%%%%%%
\begin{definition} \lb{d2.4} 
Assume Hypothesis \ref{h2.1}. \\[1mm]  
In case $(i)$ in Theorem \ref{t2.3}, $\tau$ is said to be in the {\it limit circle case} at $b$ $($resp., at $a$$)$. $($Frequently, $\tau$ is then called {\it quasi-regular} at $b$ $($resp., $a$$)$.$)$
\\[1mm] 
In case $(ii)$ in Theorem \ref{t2.3}, $\tau$ is said to be in the {\it limit point case} at $b$ $($resp., at $a$$)$. \\[1mm]
If $\tau$ is in the limit circle case at $a$ and $b$ then $\tau$ is also called {\it quasi-regular} on $(a,b)$. 
\end{definition}
%%%%%%%%%%%%%

%%%%%%%%%%%%%%%
\begin{theorem} \lb{t2.5}
Assume Hypothesis~\ref{h2.1}, then the following items $(i)$ and $(ii)$ hold: \\[1mm] 
$(i)$ If $\tau$ is in the limit point case at $a$ $($resp., $b$$)$, then 
\begin{equation} 
W(f,g)(a)=0 \, \text{$($resp., $W(f,g)(b)=0$$)$ for all $f,g\in\dom(T_{max})$.} 
\end{equation} 
$(ii)$ Let $T_{min}=\ol{\dot T_{min} }$. Then
\begin{align}
\begin{split}
n_\pm(T_{min}) &= \dim(\ker(T_{max} \mp i I))    \\
& = \begin{cases}
2 & \text{if $\tau$ is in the limit circle case at $a$ and $b$,}\\
1 & \text{if $\tau$ is in the limit circle case at $a$} \\
& \text{and in the limit point case at $b$, or vice versa,}\\
0 & \text{if $\tau$ is in the limit point case at $a$ and $b$}.
\end{cases}
\end{split} 
\end{align}
In particular, $T_{min} = T_{max}$ is self-adjoint $\big($i.e., $\dot T_{min}$ is essentially self-adjoint\,$\big)$ if and only if $\tau$ is in the limit point case at $a$ and $b$. 
\end{theorem}
%%%%%%%%%%%

Next, we turn to a description of all self-adjoint extensions of $T_{min}$.

%%%%%%%%%%%%%
\begin{theorem} \lb{t2.6}
Assume Hypothesis \ref{h2.1} and that $\tau$ is in the limit circle case at $a$ and $b$ $($i.e., $\tau$ is quasi-regular 
on $(a,b)$$)$. In addition, assume that 
$v_j \in \dom(T_{max})$, $j=1,2$, satisfy 
\begin{equation}
W(\ol{v_1}, v_2)(a) = W(\ol{v_1}, v_2)(b) = 1, \quad W(\ol{v_j}, v_j)(a) = W(\ol{v_j}, v_j)(b) = 0, \; j= 1,2.  
\end{equation}
$($E.g., real-valued solutions $v_j$, $j=1,2$, of $(\tau - \lambda) u = 0$ with $\lambda \in \bbR$, such that 
$W(v_1,v_2) = 1$.$)$ For $g\in\dom(T_{max})$ we introduce the generalized boundary values 
\begin{align}
\begin{split} 
\wti g_1(a) &= - W(v_2, g)(a), \quad \wti g_1(b) = - W(v_2, g)(b),    \\
\wti g_2(a) &= W(v_1, g)(a), \quad \;\,\,\, \wti g_2(b) = W(v_1, g)(b).   \lb{2.10}
\end{split} 
\end{align}
Then the following items $(i)$--$(iii)$ hold: \\[1mm] 
$(i)$ All self-adjoint extensions $T_{\ga,\de}$ of $T_{min}$ with separated boundary conditions are of the form
\begin{align}
& T_{\ga,\de} f = \tau f, \quad \ga,\de\in[0,\pi),   \no \\
& f \in \dom(T_{\ga,\de})=\big\{g\in\dom(T_{max}) \, \big| \, \sin(\ga) \wti g_2(a) + \cos(\ga) \wti g_1(a) = 0;   \lb{2.11} \\ 
& \hspace*{5.5cm} \, \sin(\de) \wti g_2(b) + \cos(\de) \wti g_1(b) = 0 \big\}.    \no 
\end{align}
$(ii)$ All self-adjoint extensions $T_{\varphi,R}$ of $T_{min}$ with coupled boundary conditions are of the type
\begin{align}
\begin{split} 
& T_{\varphi,R} f = \tau f,    \\
& f \in \dom(T_{\varphi,R})=\bigg\{g\in\dom(T_{max}) \, \bigg| \begin{pmatrix} \wti g_1(b) \\ \wti g_2(b)\end{pmatrix} 
= e^{i\varphi}R \begin{pmatrix}
\wti g_1(a)\\ \wti g_2(a)\end{pmatrix} \bigg\}, \lb{2.11a}
\end{split}
\end{align}
where $\varphi\in[0,2\pi)$, and $R$ is a real $2\times2$ matrix with $\det(R)=1$ 
$($i.e., $R \in SL(2,\bbR)$$)$.  \\[1mm] 
$(iii)$ Every self-adjoint extension of $T_{min}$ is either of type $(i)$ $($i.e., separated\,$)$ or of type 
$(ii)$ $($i.e., coupled\,$)$.
\end{theorem}
%%%%%%%%%%%%%%

%%%%%%%
\begin{remark} \lb{r2.7}
$(i)$ If $\tau$ is in the limit point case at one endpoint, say, at the endpoint $b$, one omits the corresponding boundary condition involving $\delta \in [0, \pi)$ at $b$ in \eqref{2.11} to obtain all self-adjoint extensions $T_{\gamma}$ of 
$T_{min}$, indexed by $\gamma \in [0, \pi)$. In the case where $\tau$ is in the limit point case at both endpoints, all boundary values and boundary conditions become superfluous as in this case $T_{min} = T_{max}$ is self-adjoint. \\[1mm] 
$(ii)$ Assume the special case where $\tau$ is regular on the finite interval $[a,b]$, that is, suppose that Hypothesis \ref{h2.1} is replaced by the more stringent set of assumptions: \\[1mm] 
{\bf Hypothesis} ($\tau$ regular on $[a,b]$.) \\[1mm] 
Let $(a,b) \subset \bbR$ be a finite interval and suppose that $p,q,r$ are $($Lebesgue\,$)$ measurable functions on $(a,b)$  
such that the following items $(i')$--$(iii')$ hold: \\[1mm] 
$(i')$ $r > 0$ a.e.~on $(a,b)$, $r \in L^1((a,b);dx)$. \\[1mm]
$(ii')$ $p > 0$ a.e.~on $(a,b)$, $1/p \in L^1((a,b);dx)$. \\[1mm]
$(iii')$ $q$ is real-valued a.e.~on $(a,b)$, $q \in L^1((a,b);dx)$.  \\[1mm] 
\indent 
In this case one chooses $v_j \in \dom(T_{max})$, $j=1,2$, 
such that 
\begin{align}
v_1(x) = \begin{cases} \theta_0(\lambda,x,a), & \text{for $x$ near $a$}, \\
\theta_0(\lambda,x,b), & \text{for $x$ near $b$},  \end{cases}   \quad 
v_2(x) = \begin{cases} \phi_0(\lambda,x,a), & \text{for $x$ near $a$}, \\
\phi_0(\lambda,x,b), & \text{for $x$ near $b$},  \end{cases}   \lb{2.12}
\end{align} 
where $\phi_0(\lambda,\, \cdot \,,d)$, $\theta_0(\lambda,\, \cdot \,,d)$, $d \in \{a,b\}$, are real-valued solutions of $(\tau - \lambda) u = 0$, $\lambda \in \bbR$, satisfying the boundary conditions 
\begin{align}
\begin{split} 
& \phi_0(\lambda,a,a) = \theta_0^{[1]}(\lambda,a,a) = 0, \quad \theta_0(\lambda,a,a) = \phi_0^{[1]}(\lambda,a,a) = 1, \\ 
& \phi_0(\lambda,b,b) = \theta_0^{[1]}(\lambda,b,b) = 0, \quad \; \theta_0(\lambda,b,b) = \phi_0^{[1]}(\lambda,b,b) = 1. 
\lb{2.13}
\end{split} 
\end{align} 
Then one verifies that
\begin{align}
\wti g_1 (a) = g(a), \quad \wti g_1 (b) = g(b), \quad \wti g_2 (a) = g^{[1]}(a), \quad \wti g_2 (b) = g^{[1]}(b),   \lb{2.14} 
\end{align}
and hence Theorem \ref{t2.6} in the special regular case recovers the well-known situation of separated self-adjoint boundary conditions for three-coefficient regular Sturm--Liouville operators in $\Lr$. 
\\[1mm]
$(iii)$ In connection with \eqref{2.10}, an explicit calculation demonstrates that for $g, h \in \dom(T_{max})$,
\begin{equation}
\wti g_1(d) \wti h_2(d) - \wti g_2(d) \wti h_1(d) = W(g,h)(d), \quad d \in \{a,b\},   \lb{2.15}
\end{equation} 
interpreted in the sense that either side in \eqref{2.15} has a finite limit as $d \downarrow a$ and $d \uparrow b$. 
Of course, for \eqref{2.15} to hold at $d \in \{a,b\}$, it suffices that $g$ and $h$ lie locally in $\dom(T_{max})$ near $x=d$.  \hfill $\diamond$
\end{remark} 
%%%%%%%

In the special case where $T_{min}$ is bounded from below, one can further analyze the generalized boundary values  \eqref{2.10} in the singular context by invoking principal and nonprincipal solutions of $\tau u = \lambda u$ for appropriate $\lambda \in \bbR$. This leads to natural analogs of \eqref{2.14} also in the singular case, and we will turn to this topic next. 

We start by reviewing some oscillation theory with particular emphasis on principal and nonprincipal solutions, a notion originally due to Leighton and Morse \cite{LM36}, Rellich \cite{Re43}, \cite{Re51}, and Hartman and Wintner \cite[Appendix]{HW55} (see also \cite{CGN16}, \cite[Sects~13.6, 13.9, 13.0]{DS88}, 
\cite[Ch.~XI]{Ha02}, \cite{NZ92}, \cite[Chs.~4, 6--8]{Ze05}). 

%%%%%%%%%%%%%%
\begin{definition} \lb{d2.8}
Assume Hypothesis \ref{h2.1}. \\[1mm] 
$(i)$ Fix $c\in (a,b)$ and $\lambda\in\bbR$. Then $\tau - \lam$ is
called {\it nonoscillatory} at $a$ $($resp., $b$$)$, 
if every real-valued solution $u(\lambda,\dott)$ of 
$\tau u = \lambda u$ has finitely many
zeros in $(a,c)$ $($resp., $(c,b)$$)$. Otherwise, $\tau - \lam$ is called {\it oscillatory}
at $a$ $($resp., $b$$)$. \\[1mm] 
$(ii)$ Let $\lambda_0 \in \bbR$. Then $T_{min}$ is called bounded from below by $\lambda_0$, 
and one writes $T_{min} \geq \lambda_0 I$, if 
\begin{equation} 
(u, [T_{min} - \lambda_0 I]u)_{L^2((a,b);rdx)}\geq 0, \quad u \in \dom(T_{min}).
\end{equation}
\end{definition}
%%%%%%%%%%%%%%

The following is a key result. 

%%%%%%%%%%%%%%
\begin{theorem} \lb{t2.9} 
Assume Hypothesis \ref{h2.1}. Then the following items $(i)$--$(iii)$ are
equivalent: \\[1mm] 
$(i)$ $T_{min}$ $($and hence any symmetric extension of $T_{min})$
is bounded from below. \\[1mm] 
$(ii)$ There exists a $\nu_0\in\bbR$ such that for all $\lambda < \nu_0$, $\tau - \lam$ is
nonoscillatory at $a$ and $b$. \\[1mm]
$(iii)$ For fixed $c, d \in (a,b)$, $c \leq d$, there exists a $\nu_0\in\bbR$ such that for all
$\lambda<\nu_0$, $\tau u = \lambda u$ has $($real-valued\,$)$ nonvanishing solutions
$u_a(\lambda,\dott) \neq 0$,
$\hatt u_a(\lambda,\dott) \neq 0$ in the neighborhood $(a,c]$ of $a$, and $($real-valued\,$)$ nonvanishing solutions
$u_b(\lambda,\dott) \neq 0$, $\hatt u_b(\lambda,\dott) \neq 0$ in the neighborhood $[d,b)$ of
$b$, such that 
\begin{align}
&W(\hatt u_a (\lambda,\dott),u_a (\lambda,\dott)) = 1,
\quad u_a (\lambda,x)=\oh(\hatt u_a (\lambda,x))
\text{ as $x\downarrow a$,} \lb{2.17} \\
&W(\hatt u_b (\lambda,\dott),u_b (\lambda,\dott))\, = 1,
\quad u_b (\lambda,x)\,=\oh(\hatt u_b (\lambda,x))
\text{ as $x\uparrow b$,} \lb{2.18} \\
&\int_a^c dx \, p(x)^{-1}u_a(\lambda,x)^{-2}=\int_d^b dx \, 
p(x)^{-1}u_b(\lambda,x)^{-2}=\infty,  \lb{2.19} \\
&\int_a^c dx \, p(x)^{-1}{\hatt u_a(\lambda,x)}^{-2}<\infty, \quad 
\int_d^b dx \, p(x)^{-1}{\hatt u_b(\lambda,x)}^{-2}<\infty. \lb{2.20}
\end{align}
\end{theorem}
%%%%%%%%%%%%%%

%%%%%%%%%%%%%% 
\begin{definition} \lb{d2.10}
Assume Hypothesis \ref{h2.1}, suppose that $T_{min}$ is bounded from below, and let 
$\lambda\in\bbR$. Then $u_a(\lambda,\dott)$ $($resp., $u_b(\lambda,\dott)$$)$ in Theorem
\ref{t2.9}\,$(iii)$ is called a {\it principal} $($or {\it minimal}\,$)$
solution of $\tau u=\lambda u$ at $a$ $($resp., $b$$)$. A real-valued solution 
$\wti{\wti u}_a(\lambda,\dott)$ $($resp., $\wti{\wti u}_b(\lambda,\dott)$$)$ of $\tau
u=\lambda u$ linearly independent of $u_a(\lambda,\dott)$ $($resp.,
$u_b(\lambda,\dott)$$)$ is called {\it nonprincipal} at $a$ $($resp., $b$$)$.
\end{definition}
%%%%%%%%%%%%%%

Principal and nonprincipal solutions are well-defined due to Lemma \ref{l2.11} below. 

%%%%%%%%%%%%%%%
\begin{lemma} \lb{l2.11} Assume Hypothesis \ref{h2.1} and suppose that $T_{min}$ is bounded 
from below. Then $u_a(\lambda,\dott)$ and $u_b(\lambda,\dott)$ in Theorem
\ref{t2.9}\,$(iii)$ are unique up to $($nonvanish- ing\,$)$ real constant multiples. Moreover,
$u_a(\lambda,\dott)$ and $u_b(\lambda,\dott)$ are minimal solutions of
$\tau u=\lambda u$ in the sense that 
\begin{align}
u(\lambda,x)^{-1} u_a(\lambda,x)&=\oh(1) \text{ as $x\downarrow a$,} 
\lb{2.21} \\ 
u(\lambda,x)^{-1} u_b(\lambda,x)&=\oh(1) \text{ as $x\uparrow b$,} \lb{2.22}
\end{align}
for any other solution $u(\lambda,\dott)$ of $\tau u=\lambda u$
$($which is nonvanishing near $a$, resp., $b$$)$ with
$W(u_a(\lambda,\dott),u(\lambda,\dott))\neq 0$, respectively, 
$W(u_b(\lambda,\dott),u(\lambda,\dott))\neq 0$. 
\end{lemma}
%%%%%%%%%%%%

Given these oscillation theoretic preparations, one can now revisit and complement Theorem \ref{t2.6} as follows: 

%%%%%%%%
\begin{theorem} \lb{t2.12}
Assume Hypothesis \ref{h2.1} and that $\tau$ is in the limit circle case at $a$ and $b$ $($i.e., $\tau$ is quasi-regular 
on $(a,b)$$)$. In addition, assume that $T_{min} \geq \lambda_0 I$ for some $\lambda_0 \in \bbR$, and denote by 
$u_a(\lambda_0, \dott)$ and $\hatt u_a(\lambda_0, \dott)$ $($resp., $u_b(\lambda_0, \dott)$ and 
$\hatt u_b(\lambda_0, \dott)$$)$ principal and nonprincipal solutions of $\tau u = \lambda_0 u$ at $a$ 
$($resp., $b$$)$, satisfying
\begin{equation}
W(\hatt u_a(\lambda_0,\dott), u_a(\lambda_0,\dott)) = W(\hatt u_b(\lambda_0,\dott), u_b(\lambda_0,\dott)) = 1.  
\lb{2.23} 
\end{equation}
Then the following items $(i)$--$(iii)$ hold: \\[1mm]  
$(i)$ Introducing $v_j \in \dom(T_{max})$, $j=1,2$, via 
\begin{align}
v_1(x) = \begin{cases} \hatt u_a(\lambda_0,x), & \text{for $x$ near a}, \\
\hatt u_b(\lambda_0,x), & \text{for $x$ near b},  \end{cases}   \quad 
v_2(x) = \begin{cases} u_a(\lambda_0,x), & \text{for $x$ near a}, \\
u_b(\lambda_0,x), & \text{for $x$ near b},  \end{cases}   \lb{2.24}
\end{align} 
one obtains for all $g \in \dom(T_{max})$, 
\begin{align}
\begin{split} 
\wti g(a) &= - W(v_2, g)(a) = \wti g_1(a) =  - W(u_a(\lambda_0,\dott), g)(a)    \\
&= \lim_{x \downarrow a} \f{g(x)}{\hatt u_a(\lambda_0,x)},    \\
\wti g(b) &= - W(v_2, g)(b) = \wti g_1(b) =  - W(u_b(\lambda_0,\dott), g)(b)   \\
&= \lim_{x \uparrow b} \f{g(x)}{\hatt u_b(\lambda_0,x)},    
\lb{2.25} 
\end{split} \\
\begin{split} 
{\wti g}^{\, \prime}(a) &= W(v_1, g)(a) = \wti g_2(a) = W(\hatt u_a(\lambda_0,\dott), g)(a)   \\
&= \lim_{x \downarrow a} \f{g(x) - \wti g(a) \hatt u_a(\lambda_0,x)}{u_a(\lambda_0,x)},    \\ 
{\wti g}^{\, \prime}(b) &= W(v_1, g)(b) = \wti g_2(b) = W(\hatt u_b(\lambda_0,\dott), g)(b)   \\ 
&= \lim_{x \uparrow b} \f{g(x) - \wti g(b) \hatt u_b(\lambda_0,x)}{u_b(\lambda_0,x)}.    \lb{2.26}
\end{split} 
\end{align}
In particular, the limits on the right-hand sides in \eqref{2.25}, \eqref{2.26} exist. \\[1mm] 
$(ii)$ All self-adjoint extensions $T_{\gamma,\delta}$ of $T_{min}$ with separated boundary conditions are of the form
\begin{align}
& T_{\gamma,\delta} f = \tau f, \quad \gamma,\delta \in[0,\pi),   \no \\
& f \in \dom(T_{\gamma,\delta})=\big\{g\in\dom(T_{max}) \, \big| \, \sin(\gamma) {\wti g}^{\, \prime}(a) 
+ \cos(\gamma) \wti g(a) = 0;   \lb{2.27} \\ 
& \hspace*{5.5cm} \,  \sin(\delta) {\wti g}^{\, \prime}(b) + \cos(\delta) \wti g(b) = 0 \big\}.    \no 
\end{align}
Moreover, $\sigma(T_{\gamma,\delta})$ is simple. \\[1mm]
$(iii)$ All self-adjoint extensions $T_{\varphi,R}$ of $T_{min}$ with coupled boundary conditions are of the type
\begin{align}
\begin{split} 
& T_{\varphi,R} f = \tau f,    \\
& f \in \dom(T_{\varphi,R})=\bigg\{g\in\dom(T_{max}) \, \bigg| \begin{pmatrix} \wti g(b) 
\\ {\wti g}^{\, \prime}(b) \end{pmatrix} = e^{i\varphi}R \begin{pmatrix}
\wti g(a) \\ {\wti g}^{\, \prime}(a) \end{pmatrix} \bigg\}, \lb{2.27a}
\end{split}
\end{align}
where $\varphi\in[0,2\pi)$, and $R \in SL(2,\bbR)$.
\end{theorem}
%%%%%%%%

Moreover, under the hypotheses of Theorem \ref{t2.12}, relation \eqref{2.5} implies that the minimal operator takes on the form
\begin{align}
\begin{split} 
& T_{min} f = \tau f, \\
& f \in \dom(T_{min})= \big\{g\in\dom(T_{max})  \, \big| \, \wti g(a) = {\wti g}^{\, \prime}(a) =0 
= \wti g(b) = {\wti g}^{\, \prime}(b)\big\}.      \lb{2.27b} 
\end{split}
\end{align}

The Friedrichs extension $T_F$ of $T_{min}$ now permits a particularly simple characterization in terms of the generalized boundary values $\wti g(a), \wti g(b)$ as derived by Kalf  \cite{Ka78} and subsequently by Niessen and Zettl \cite{NZ92} (see also \cite{Re51}, \cite{Ro85} and the extensive literature cited in \cite{GLN20}, 
\cite[Ch.~13]{GZ21}):

%%%%%%%%
\begin{theorem} \lb{t2.13}
Assume Hypothesis \ref{h2.1} and that $\tau$ is in the limit circle case at $a$ and $b$ $($i.e., $\tau$ 
is quasi-regular on $(a,b)$$)$. In addition, assume that $T_{min} \geq \lambda_0 I$ for some $\lambda_0 \in \bbR$. Then the Friedrichs extension $T_F$ of $T_{min}$ is characterized by
\begin{equation}
T_F f = \tau f, \quad f \in \dom(T_F)= \big\{g\in\dom(T_{max})  \, \big| \, \wti g(a) = 0 = \wti g(b)\big\}.    \lb{2.28}
\end{equation}
In particular, $T_F = T_{0,0}$.
\end{theorem}
%%%%%%%%

%%%%%%%%
\begin{remark} \lb{r2.14}
$(i)$ As in \eqref{2.15}, one readily verifies for $g, h \in \dom(T_{max})$,
\begin{equation}
\wti g(d) {\wti h}^{\, \prime}(d) - {\wti g}^{\, \prime}(d) \wti h(d) = W(g,h)(d), \quad d \in \{a,b\},    \lb{2.29} 
\end{equation} 
again interpreted in the sense that either side in \eqref{2.29} has a finite limit as $d \downarrow a$ and 
$d \uparrow b$. In particular, if $\tau$ is regular at an endpoint then Remark \ref{r2.7}\,$(ii)$ shows that the generalized boundary values in \eqref{2.25}, \eqref{2.26} reduce to the canonical ones in \eqref{2.14}. \\[1mm] 
$(ii)$ While the principal solution at an endpoint is unique up to constant multiples (which we will ignore), nonprincipal solutions differ by additive constant multiples of the principal solution. As a result, if 
\begin{align}
\begin{split} 
& \hatt u_a (\lambda_0, \dott) \longrightarrow \hatt u_a (\lambda_0, \dott) + C u_a(\lambda_0,\dott), \quad C \in \bbR,    \\
& \quad \text{then } \, \wti g(a) \longrightarrow \wti g(a), 
\quad {\wti g}^{\, \prime}(a) \longrightarrow {\wti g}^{\, \prime}(a) - C \wti g(a),    
\end{split} 
\end{align}
and analogously at the endpoint $b$. Hence, generalized boundary values ${\wti g}^{\, \prime}(d)$ at the endpoint 
$d \in \{a,b\}$ depend on the choice of nonprincipal solution $\hatt u_{d}(\lambda_0, \dott)$ of $\tau u = \lambda_0 u$ at $d$. However, the Friedrichs boundary conditions $\wti g(a) = 0 = \wti g(b)$ are clearly independent of the choice of 
nonprincipal solution. \\[1mm]
$(iii)$ As always in this context, if $\tau$ is in the limit point case at one or both interval endpoints, the corresponding boundary conditions at that endpoint are dropped and only a separated boundary condition at the other end point (if the latter is a limit circle endpoint for $\tau$) has to be imposed in Theorems \ref{t2.12} and \ref{t2.13}. In other words, the generalized boundary values \eqref{2.10} and \eqref{2.25}, \eqref{2.26} are only relevant if the endpoint in question is of the  limit circle type. In the case where $\tau$ is in the limit point case at both endpoints, all boundary values and boundary conditions become superfluous as in this case $T_{min} = T_{max}$ is self-adjoint. 
\hfill $\diamond$
\end{remark}
%%%%%%%%

All results surveyed in this section can be found in \cite{GLN20} and \cite[Ch.~13]{GZ21} which contain very detailed 
lists of references to the basics of Weyl--Titchmarsh theory. Here we just mention a few additional and classical 
sources such as \cite[Sect.~129]{AG81}, \cite[Chs.~8, 9]{CL85}, \cite[Sects.~13.6, 13.9, 13.10]{DS88}, 
\cite[Ch.~III]{JR76}, \cite[Ch.~V]{Na68}, \cite{NZ92}, \cite[Ch.~6]{Pe88}, \cite[Ch.~9]{Te14}, \cite[Sect.~8.3]{We80}, \cite[Ch.~13]{We03}, \cite[Chs.~4, 6--8]{Ze05}.

%%%%%%%%%%%%%%%%%%%%%%%%%%%%%%
%%%%%%%%%%%%%%%%%%%%%%%%%%%%%%
\section{The Krein--von Neumann extension of $T_{min} > 0$} \lb{s3}
%%%%%%%%%%%%%%%%%%%%%%%%%%%%%%
%%%%%%%%%%%%%%%%%%%%%%%%%%%%%%

In this section we derive the Krein--von Neumann extension for $T_{min}$ under the assumption 
$T_{min} \geq \varepsilon I$ for some $\varepsilon > 0$. We continue with some more abstract facts on the 
Krein--von Neumann extension of strictly positive symmetric operators in a complex Hilbert space and refer to 
\cite[Sect.\ 109]{AG81}, \cite{AS80}, \cite{AN70}, \cite{AHSD01}, \cite{AT05}, \cite{AT09}, \cite{AGMT10}, \cite{AGMST10},  \cite{AGMST13}, \cite{AGLMS17}, \cite[Sect.~5.4]{BHS20}, 
\cite{Bi56}, \cite{DM91}, \cite{DM95}, \cite[Part III]{Fa75}, \cite{Gr83}, \cite[Sect.\ 13.2]{Gr09},  \cite{HMD04}, \cite{HSDW07}, \cite{Kr47}, \cite{Kr47a}, 
\cite{Ne83}, \cite{PS96}, \cite{SS03}, \cite{Si98}, \cite{Sk79}, \cite{St96}, \cite{St73}, \cite{Ts81}, \cite{Vi63}, \cite{vN29}, 
and the references cited therein for some of the basic literature in this context. 

Denote by 
\begin{equation}
n_{\pm}(T) = \dim(\ker(T^* \mp i I_{\cH})) = \dim\big(\ran(T\pm i I_{\cH})^{\bot}\big) \in \bbN_0 \cup \{\infty\}, 
\end{equation}
the deficiency indices of a densely defined, closed, symmetric operator $T$ in the complex, separable Hilbert space $\cH$. 

The Krein--von Neumann and Friedrichs extension of a densely defined, closed, symmetric operator $S$ with 
$n_{\pm}(S) > 0$, satisfying  
\begin{equation} 
S \geq 0,  
\end{equation} 
are denoted by $S_K$ and $S_F$, respectively. If, in addition, 
\begin{equation} 
S \geq \varepsilon I_{\cH}    \lb{3.3}
\end{equation} 
for some $\varepsilon>0$, then one also has 
\begin{align} 
& n_{\pm}(S) = \dim(\ker(S^*)),    \lb{3.4} \\
& S_F  \geq \varepsilon I_{\cH},       \lb{3.4a}
\end{align} 
and 
\begin{align}\lb{3.5}
\begin{split} 
& \dom(S_K)=\dom(S) \dotplus \ker(S^*),      \\
& \, S_K f = S^* f, \quad f \in \dom(S_K).   
\end{split} 
\end{align}
For completeness we also recall that under hypothesis \eqref{3.3} 
\begin{align}
& \dom(S_F) = \dom(S) \dotplus (S_F)^{-1} \ker(S^*), \quad S_F f = S^* f, \quad f \in \dom(S_F),   \\
& \dom(S^*) = \dom(S) \dotplus (S_F)^{-1} \ker(S^*) \dotplus \ker(S^*),   \\
& \ker(S_K) = \ker(S^*).     \lb{3.8}
\end{align}

Here the notation $\dotplus$ addresses the direct (not orthogonal direct) sum in the sense that if $X_1,\ X_2$ are linear subspaces of a Banach space $X$, then $X_1\dotplus X_2$ denotes the subspace of $X$ given by
\begin{align}
X_1\dotplus X_2=\{x\in X| x=x_1+x_2,\ x_j\in X_j,\ j=1,2\},
\end{align}
assuming
\begin{align}
X_1\cap X_2=\{0\}.
\end{align}

%%%%%%%%%%%%%%
\begin{remark} \lb{r3.1} 
If $S\geq \varepsilon I_{\cH},\ \varepsilon>0$, then $\dom(S)\dotplus\ker(S^*)$ is well-defined since
\begin{align}
f\in\dom(S)\cap\ker(S^*)
\end{align}
implies $0=S^* f=Sf$, and hence $f=0$ as $S\geq\varepsilon I_{\cH},\ \varepsilon>0$. \hfill $\diamond$
\end{remark}
%%%%%%%%%%%%%%

%%%%%%%%%%%%%%
\begin{lemma}\lb{l3.2}
Suppose $S$ is densely defined, symmetric, and for some $\varepsilon > 0$, $S \geq \varepsilon I_{\cH}$. If 
$\wti S$ is a self-adjoint extension of $S$ such that $\dom\big(\wti S\big)\supset \ker(S^*)$, then $\wti S=S_K$.
\end{lemma}
%%%%%%%%%%%%%%
\begin{proof}
Without loss of generality we may assume that $S$ is closed. By definition, $\dom\big(\wti S\big)\supset \dom(S)$, and by assumption, $\dom\big(\wti S\big)\supset \ker(S^*)$, hence 
$\dom\big(\wti S\big)\supseteq \dom(S)\dotplus\ker(S^*)=\dom(S_K)$. Utilizing the facts $\wti S = S^*\big|_{\dom(\wti S)}$ and $S_K = S^*\big|_{\dom(S_K)}$, one infers that
\begin{equation}
\wti S\big|_{\dom(S_K)} = S^*\big|_{\dom(\wti S) \cap \dom(S_K)} = S^*\big|_{\dom(S_K)} = S_K, 
\end{equation}
that is, $\wti S \supseteq S_K$. Since self-adjoint operators are maximal in the sense that $S_K$ has no proper symmetric extensions, one concludes that $\wti S=S_K$. 
\end{proof}
%%%%%%%%%%%%%%

%%%%%%%%%%%%%%
\begin{remark} \lb{r3.3} 
$(i)$ From the outset, \eqref{3.5} implies \eqref{3.8}, that is, $\ker(S_K)=\ker(S^*)$.  
Lemma \ref{l3.2} now implies a converse in the sense that if $\ker(\wti S)=\ker(S^*)$ (and hence lies in 
$\dom\big(\wti S\big)$), then $\wti S=S_K$. One notes that Lemma \ref{l3.2} does not {\it a priori} assume that $\wti S$ is bounded from below. \\[1mm] 
$(ii)$ The fact \eqref{3.8} has been isolated in \cite{AS80} as uniquely identifying the Krein--von Neumann extension (under the hypothesis $S \geq \varepsilon I_{\cH}$) (see also \cite[eq.~(2.39)]{AGMT10}), as a byproduct of an entirely different quadratic form approach that characterizes all nonnegative self-adjoint extensions of $S$. Here we derive this uniqueness aspect with entirely elementary means only. \\[1mm]  
$(iii)$ Applications to $2m$th order regular differential operators: Suppose 
$T_{min}=\overline{\tau_{2m}\big|_{C_0^\infty ((a,b))}}$, then $T^*_{min}=T_{max}$, that is, no boundary conditions are necessary in $\dom(T_{max})$. Then $\wti T$, a self-adjoint extension of $T_{min}$ in $L^2((a,b);rdx)$ equals $T_{min,K}$, the Krein extension of $T_{min}$ in $L^2((a,b);rdx)$, if and only if
\begin{align}\lb{3.13}
\dim\big(\ker\big(\wti T\big)\big)=2m.
\end{align}
Of course, $\dim(\ker(T_{max}))=2m$ since we assumed $T_{min}$ to be regular. \\[1mm] 
$(iv)$ Relation \eqref{3.13} only holds if $T_{min}$ is indeed minimally defined, that is, as the closure of $\tau_{2m}\big|_{C_0^\infty ((a,b))}$. If some of the possible zero boundary conditions are missing in $\dom(T_{min})$ then they will reappear in $\dom(T_{max} = T_{min}^*)$ and hence $2m$ in \eqref{3.13} has to be diminished accordingly. \\[1mm] 
$(v)$ The $2m$ solutions giving rise to \eqref{3.13} (in the regular case) are simply generated by solving the ordinary differential equation of $2m$th order
\begin{align} \lb{3.14}
\tau_{2m} y=0
\end{align}
in the distributional sense. \hfill$\diamond$
\end{remark}
%%%%%%%%%%%%%%

To demonstrate that $\varepsilon>0$ is necessary for Lemma \ref{l3.2} to hold, we recall the following (counter) example.

%%%%%%%%%%%%%%
\begin{example} \lb{e3.4} 
Let $\cH=L^2((0,\infty);dx)$, and
\begin{align}
& T_{min}^{(0)} f = - f'',   \no \\
& f \in \dom\big(T_{min}^{(0)}\big) = \big\{g\in L^2((0,\infty);dx) \, \big| \, \text{for all $R>0$:} \, g,g' \in AC([0,R]);   \\ 
& \hspace*{5.35cm} g(0)=g'(0)=0; \, g''\in L^2((0,\infty);dx)\big\},  \no \\
&T_{min}^{(0)}=\overline{-d^2/dx^2\big|_{C_0^\infty ((0,\infty))}}.
\end{align}
Then
\begin{equation}
T_{min}^{(0)} \geq 0
\end{equation}
but there is no $\varepsilon > 0$ such that $T_{min}^{(0)} \geq \varepsilon I$. Moreover, 
\begin{align}
& T_{max}^{(0)} f =\big(T_{min}^{(0)}\big)^* f =- f'',   \no \\
& f \in \dom\big(T_{max}^{(0)}\big) = \big\{g\in L^2((0,\infty);dx) \, \big| \, \text{for all $R>0$:} \, g,g' \in AC([0,R]);    \\
& \hspace*{8.15cm} g''\in L^2((0,\infty);dx)\big\},  \no 
\end{align}
and
\begin{align} \lb{3.19}
\begin{split} 
& T_{min,K}^{(0)} f =T_N^{(0)} f =- f'',   \\
& f \in \dom\big(T_{min,K}^{(0)} = T_N^{(0)}\big) = \big\{g\in\dom\big(T_{max}^{(0)}\big) \, \big| \, g'(0)=0\big\}.
\end{split} 
\end{align}
Furthermore, consider self-adjoint extensions $T_{\a}^{(0)}$ of $T_{min}^{(0)}$ given by 
\begin{align}
\begin{split} 
& T_{\a}^{(0)} f = - f'',\quad \a\in[0,\infty)\cup\{\infty\},   \\
& f \in \dom\big(T_{\a}^{(0)}\big) = \big\{g\in\dom\big(T_{max}^{(0)}\big) \, \big| \, g'(0)=\a g(0)\big\}.  
\end{split}
\end{align}
Then 
\begin{equation} 
\ker\big(T_{\a}^{(0)}\big)=\ker\big(T_{min,K}^{(0)}=T_N^{(0)} \equiv T_{\a=0}^{(0)}\big) 
= \ker\big(\big(T_{min}^{(0)}\big)^* = T_{max}^{(0)}\big) = \{0\},
\end{equation}
and 
\begin{equation} 
\sigma\big(T_{\a}^{(0)}\big) = [0,\infty),  \quad \a \in [0,\infty) \cup \{\infty\}. 
\end{equation} 
Thus, Lemma \ref{l3.2} requires $\varepsilon>0$. For the fact \eqref{3.19} see, for example, \cite[Corollary 5.6]{GKMT01} $($choose $\gamma=\pi/2,\ q(x)=0$, and note that $m_0^W(z)=iz^{1/2}$, hence $m_0^W(0)=0$$)$.

Of course,
\begin{align}
\begin{split} 
& T_{\a=\infty}^{(0)} f = T_{min,F}^{(0)} f = T_D^{(0)} f = - f'',    \\
& f \in \dom\big(T_D^{(0)}\big) = \big\{g\in\dom\big(T_{max}^{(0)}\big) \, \big| \, g(0)=0\big\},     \lb{3.23} 
\end{split} 
\end{align}
represents the Friedrichs $($resp., Dirichlet\,$)$ extension of $T_{min}^{(0)}$. 
\end{example}
%%%%%%%%%%%%%%

Combining \cite{CGNZ14} and \cite{GLN20} one can now extend the description of the Krein extension from the 
known regular case to the singular case as follows:

%%%%%%
\begin{theorem} \lb{t3.5}
In addition to Hypothesis \ref{h2.1}, suppose that $T_{min} \geq \varepsilon I$ for some 
$\varepsilon > 0$. Then the following items $(i)$ and $(ii)$ hold: \\[1mm]
$(i)$ Assume that $n_{\pm}(T_{min}) = 1$ and denote the principal solutions of $\tau u = 0$ at $a$ and $b$ by $u_a(0,\dott)$ and $u_b(0,\dott)$, respectively. If $\tau$ is in the limit circle case at $a$ and in the limit point case at $b$, then the Krein--von Neumann extension $T_{\ga_K}$ of $T_{min}$ is given by 
\begin{align}
& T_{\ga_K}f = \tau f,   \no \\
& f \in \dom(T_{\ga_K})=\big\{g\in\dom(T_{max}) \, \big| \, \sin(\gamma_K) {\wti g}^{\, \prime}(a) 
+ \cos(\gamma_K) \wti g(a) = 0\big\},     \lb{3.24} \\
& \cot(\ga_K) = - \wti u_b^{\, \prime}(0,a) / \wti u_b(0,a), \quad \ga_K \in (0,\pi).   \no 
\end{align}
Similarly, if $\tau$ is in the limit circle case at $b$ and in the limit point case at $a$, then the Krein--von Neumann extension $T_{\de_K}$ of $T_{min}$ is given by 
\begin{align}
& T_{\de_K}f = \tau f,   \no \\
& f \in \dom(T_{\de_K})=\big\{g\in\dom(T_{max}) \, \big| \, \sin(\de_K) {\wti g}^{\, \prime}(b) 
+ \cos(\de_K) \wti g(b) = 0\big\},     \lb{3.25} \\
& \cot(\de_K) = - \wti u_a^{\, \prime}(0,b) / \wti u_a(0,b), \quad \de_K \in (0,\pi).   \no 
\end{align}
$(ii)$ Assume that $n_{\pm}(T_{min}) = 2$, that is, $\tau$ is in the limit circle case at $a$ and $b$. Then, introducing a 
basis for $\ker(T_{max})$, denoted by $u_1(0,\dott),\ u_2(0,\dott)$ as follows, 
\begin{align}
& \tau u_j(0,\dott) = 0, \quad j =1,2,   \no \\
& \wti u_1(0,a) = 0, \quad \wti u_1(0,b)=1,    \lb{3.26} \\
& \wti u_2(0,a) = 1, \quad \wti u_2(0,b)=0,    \no 
\end{align}
the Krein--von Neumann extension $T_{0,R_K}$ of $T_{min}$ is given by 
\begin{align}
\begin{split} 
& T_{0,R_K} f = \tau f,    \\
& f \in \dom(T_{0,R_K})=\bigg\{g\in\dom(T_{max}) \, \bigg| \begin{pmatrix} \wti g(b) 
\\ {\wti g}^{\, \prime}(b) \end{pmatrix} = R_K \begin{pmatrix}
\wti g(a) \\ {\wti g}^{\, \prime}(a) \end{pmatrix} \bigg\},      \lb{3.27}
\end{split}
\end{align}
where 
\begin{align}
R_K=\dfrac{1}{\wti u_1^{\, \prime}(0,a)}\begin{pmatrix} - \wti u_2^{\, \prime}(0,a) & 1\\
\wti u_1^{\, \prime}(0,a) \wti u_2^{\, \prime}(0,b) - \wti u_1^{\, \prime}(0,b) \wti u_2^{\, \prime}(0,a) & \wti u_1^{\, \prime}(0,b)
\end{pmatrix}.    \lb{3.28}
\end{align}
Alternatively, employing the nonprincipal solutions  $\hatt u_a(0,\dott)$ and $\hatt u_b(0,\dott)$ of $\tau u=0$ at $a$ and $b$, respectively, satisfying \eqref{2.23} and used to introduce the generalized boundary values \eqref{2.25}, \eqref{2.26}, $R_K$ can be characterized in terms of principal and nonprincipal solutions by 
\begin{align}
R_K=\begin{pmatrix} \wti{\hatt u}_{a}(0,b) & \wti u_{a}(0,b)\\
\wti{\hatt u}_{a}^{\, \prime}(0,b) & \wti u_{a}^{\, \prime}(0,b)
\end{pmatrix},  \lb{3.29}
\end{align}
equivalently, by 
\begin{align}
R_K=\begin{pmatrix} \wti u^{\, \prime}_b(0,a) & -\wti u_b(0,a)\\
-\wti{\hatt u}_b^{\, \prime}(0,a) & \wti{\hatt u}_b(0,a)
\end{pmatrix}.   \lb{3.30}
\end{align}
\end{theorem}
%%%%%%
\begin{proof} $(i)$ It suffices to prove \eqref{3.24}, the proof for \eqref{3.25} being entirely analogous. 
Since $\ker(T_{max})$ is one-dimensional, and $\tau$ is in the limit point case at $b$, one concludes that only the smallest solution of $\tau u = 0$ near $b$, that is, the principal solution, $u_b(0,\dott)$, can lie in $\ker(T_{max})$,
\begin{equation}
\ker(T_{max}) = \linspan \{u_b(0,\dott)\}. 
\end{equation}
One also notes that by \eqref{2.25}, \eqref{2.26}, 
\begin{equation}
 \wti u_b(0,b) = 0, \quad \wti u_b^{\, \prime}(0,b) = 1.
\end{equation} 
Since the Krein extension is now necessarily associated with separated boundary conditions, in fact, a self-adjoint boundary condition at the endpoint $a$ only, \eqref{3.8} implies that $u_b(0,\dott)$ must necessarily satisfy the boundary condition at $a$ which defines the Krein extension of $T_{min}$ in this case. Thus, the underlying boundary condition parameter $\ga_K \in [0,\pi)$ is determined via 
\begin{equation}
\sin(\ga_K) \wti u_b^{\, \prime}(0,a) + \cos(\ga_K)  \wti u_b(0,a) = 0. 
\end{equation}
If $\wti u_b(0,a) = 0$, then together with \eqref{2.28} and \eqref{3.4a} this would imply that 
$u_b(0,\dott) \in \dom(T_F = T_{0})$ in the notation of \eqref{2.27} (one notes that no boundary condition is 
required at $b$ due to the limit point property of $\tau$ at $b$). Thus, $0 \in \sigma_p(T_F)$, contradicting 
$T_F \geq \varepsilon I$. Hence 
$\ga_K > 0$, completing the proof of \eqref{3.24}. \\[1mm] 
$(ii)$ Since the limit circle case is assumed at the endpoints $a$ and $b$, all solutions $u(0,\dott)$ of $\tau u =0$ satisfy 
$u(0,\dott) \in \dom(T_{max})$ and hence the generalized boundary values \eqref{2.25}, \eqref{2.26} with 
$g = u(0,\dott)$ are well-defined. Introducing a basis for the null space of $T_{max}$ as in \eqref{3.26}, one notes via \eqref{2.29} that 
\begin{align}
W(u_1(0,\dott), u_2(0,\dott)) = - {\wti u_1}^{\, \prime}(0,a) = {\wti u_2}^{\, \prime}(0,b).    \lb{3.34} 
\end{align}
Next, introducing $R_K$ as in \eqref{3.28} and employing \eqref{3.26} and \eqref{3.34} one computes that 
\begin{align}
\det(R_K) = - \wti u_2^{\, \prime}(0,b)/\wti u_1^{\, \prime}(0,a) = 1, \, \text{ that is, } \,  R_K \in SL(2,\bbR).
\end{align}
Thus, \eqref{3.27} represents one of the self-adjoint extensions of $T_{min}$ characterized by $\varphi = 0$ and 
$R=R_K$ according to Theorem \ref{t2.12}\,$(iii)$, and hence it remains to show that this extension is precisely the 
Krein--von Neumann extension of $T_{min}$. 

For this purpose we turn to Lemma \ref{l3.2} next: Since 
\begin{equation} 
\dom(T_{0,R_K}) = \dom(T_{min}) \dotplus \ker(T_{min}^*) = \dom(T_{min}) \dotplus \ker(T_{max}),
\end{equation}   
and since $T_{min}$ is characterized by the vanishing of all generalized boundary values as depicted in \eqref{2.27b}, 
it suffices to verify that $u_1(0,\dott)$ and $u_2(0,\dott)$ both satisfy the boundary conditions in \eqref{3.27}, given \eqref{3.28} (see also Remark \ref{r3.3}\,$(iii)$ with $m=1$). This verification reduces to the following elementary computations, employing \eqref{3.26} and \eqref{3.34} once more,  
\begin{align}
\begin{split}
& R_K \begin{pmatrix} \wti u_1(0,a) \\ {\wti u_1}^{\, \prime}(0,a) \end{pmatrix} 
= \begin{pmatrix} 1 \\ {\wti u_1}^{\, \prime}(0,b) \end{pmatrix}  
= \begin{pmatrix} \wti u_1(0,b) \\ {\wti u_1}^{\, \prime}(0,b) \end{pmatrix}, \\   
& R_K \begin{pmatrix} \wti u_2(0,a) \\ {\wti u_2}^{\, \prime}(0,a) \end{pmatrix} 
= \begin{pmatrix} 0 \\ {\wti u_2}^{\, \prime}(0,b) \end{pmatrix}  
= \begin{pmatrix} \wti u_2(0,b) \\ {\wti u_2}^{\, \prime}(0,b) \end{pmatrix}, 
\end{split} 
\end{align}
completing the proof of \eqref{3.27}, \eqref{3.28}. 

Next, choosing
\begin{align}
\begin{split}
u_{1}(0,x)&=u_{a}(0, x)/\wti u_{a}(0, b),\\
u_{2}(0,x)&=\hatt u_{a}(0, x)-[\wti{\hatt u}_{a}(0,b)/\wti u_{a}(0,b)]u_{a}(0, x),
\end{split}
\end{align}
noting that $\wti u_a(0,b)\neq0$, one easily verifies that \eqref{3.26} is satisfied since
\begin{align}
\begin{split}
\wti u_{1}(0,a)&=\wti u_{a}(0, a)/\wti u_{a}(0, b)=0,\\
\wti u_{1}(0,b)&=\wti u_{a}(0, b)/\wti u_{a}(0, b)=1,\\
\wti u_{2}(0,a)&=\wti{\hatt u}_{a}(0, a)=1,\\
\wti u_{2}(0,b)&=\wti{\hatt u}_{a}(0, b)-\wti{\hatt u}_{a}(0, b)=0,
\end{split}
\end{align}
proving \eqref{3.29}. Expression \eqref{3.30} is proved analogously by using the principal and nonprincipal solutions near $x=b$ and choosing
\begin{align}
\begin{split}
u_{1}(0,x)&=\hatt u_{b}(0, x)-[\wti{\hatt u}_{b}(0,a)/\wti u_{b}(0,a)]u_{b}(0, x),\\
u_{2}(0,x)&=u_{b}(0, x)/\wti u_{b}(0, a).
\end{split}
\end{align}
\end{proof}
%%%%%%

Relations \eqref{3.27}, \eqref{3.28} extend \cite[Example~3.3]{CGNZ14} in the regular context to the singular one.

%%%%%%%%%%%%%%%%%%%%%%%%%%%%%%
%%%%%%%%%%%%%%%%%%%%%%%%%%%%%%
\section{Three Examples} \lb{s4}
%%%%%%%%%%%%%%%%%%%%%%%%%%%%%%
%%%%%%%%%%%%%%%%%%%%%%%%%%%%%%

In this section we illustrate Theorem \ref{t3.5} with three examples, including a generalized Bessel and Jacobi-type operators.

We start with the generalized Bessel operator following the analysis in \cite[Section 6]{GNS20}.

%%%%%%
\begin{example}[A Generalized Bessel Operator] \lb{e4.1} 
Let $a=0$ and $b\in(0,\infty)$ in \eqref{2.1}, and consider the concrete example 
\begin{equation}
\begin{split} \lb{4.1}
p(x)=x^\b ,\quad r(x)=x^\a, \quad q (x) = \frac{(2+\a-\b)^2\g^2-(1-\b)^2}{4}x^{\b-2}, \\
\a>-1,\ \b<1,\ \g\geq0,\ x\in(0,b).     
\end{split}
\end{equation}
Then 
\begin{align}
\begin{split}
\tau_{\a,\b,\g} = x^{-\a}\left[-\frac{d}{dx}x^\b\frac{d}{dx} +\frac{(2+\a-\b)^2\g^2-(1-\b)^2}{4}x^{\b-2}\right],\\
\a>-1,\ \b<1,\ \g\geq0,\ x\in(0,b),     \lb{4.2} 
\end{split}
\end{align}
is singular at the endpoint $0$ $($since the potential, $q$ is not integrable near $x=0$$)$ and is regular at $x=b$. Furthermore, $\tau_{\a,\b,\g}$ is in the limit circle case at $x=0$ if $0\leq\g<1$ and in the limit point case at $x=0$ when $\g\geq1$. 

Solutions of $\tau_{\a,\b,\g} u = z u$ are given by $($cf.\ \cite[No.~2.162, p.~440]{Ka61}$)$
\begin{align}
y_{1,\a,\b,\g}(z,x)&=x^{(1-\b)/2} J_{\gamma}\big(2z^{1/2} x^{(2+\a-\b)/2}/(2+\a-\b)\big),\quad \g\geq0,   \no \\
y_{2,\a,\b,\g}(z,x)&=\begin{cases}
x^{(1-\b)/2} J_{-\gamma}\big(2z^{1/2} x^{(2+\a-\b)/2}/(2+\a-\b)\big), & \g\notin\bbN_0,\\
x^{(1-\b)/2} Y_{\g}\big(2z^{1/2} x^{(2+\a-\b)/2}/(2+\a-\b)\big), & \g\in\bbN_0,
\end{cases}\ \g\geq0,
\end{align}
where $J_{\nu}(\dott), Y_{\nu}(\dott)$ are the standard Bessel functions of order $\nu \in \bbR$ 
$($cf.\ \cite[Ch.~9]{AS72}$)$.

In the following we assume that 
\begin{equation}
\gamma \in [0,1) 
\end{equation}
to ensure the limit circle case at $x=0$. In this case it suffices to focus on  the generalized boundary values at the singular endpoint $x = 0$. For this purpose we introduce principal and nonprincipal solutions $u_{0,\a,\b,\gamma}(0, \dott)$ and $\hatt u_{0,\a,\b,\gamma}(0, \dott)$ of $\tau_{\alpha,\beta,\gamma} u = 0$ at $x=0$ by
\begin{align}
\begin{split} 
u_{0,\a,\b,\gamma}(0, x) &= (1-\b)^{-1}x^{[1-\b+(2+\a-\b)\g]/2}, \quad \gamma \in [0,1),   \\
\hatt u_{0,\a,\b,\gamma}(0, x) &= \begin{cases} (1-\b)[(2+\a-\b) \gamma]^{-1} x^{[1-\b-(2+\a-\b)\g]/2}, & \gamma \in (0,1),     \lb{4.6} \\
(1-\b)x^{(1-\b)/2} \ln(1/x), & \gamma =0,  \end{cases}\\
&\hspace*{4.5cm} \a>-1,\; \b<1,\; x \in (0,1).
\end{split} 
\end{align} 
The generalized boundary values for $g \in \dom(T_{max,\alpha, \beta, \gamma})$ at $x=0$ are then of the form
\begin{align}
\wti g(0) &= - W(u_{0,\alpha,\beta,\gamma}(0, \dott), g)(0)     \no \\
&= \begin{cases} \lim_{x \downarrow 0} g(x)\big/\big[(1-\b)[(2+\a-\b) \gamma]^{-1} x^{[1-\b-(2+\a-\b)\g]/2}\big], & 
\gamma \in (0,1), \\[1mm]
\lim_{x \downarrow 0} g(x)\big/\big[(1-\b)x^{(1-\b)/2} \ln(1/x)\big], & \gamma =0, 
\end{cases}  \\
\wti g^{\, \prime} (0) &= W(\hatt u_{0,\alpha,\beta,\gamma}(0, \dott), g)(0)   \no \\
&= \begin{cases} \lim_{x \downarrow 0} \big[g(x) - \wti g(0) (1-\b)[(2+\a-\b) \gamma]^{-1} x^{[1-\b-(2+\a-\b)\g]/2}\big]\\
\qquad\quad \times \big[(1-\b)^{-1}x^{[1-\b+(2+\a-\b)\g]/2}\big]^{-1}, 
& \hspace{-.2cm}\gamma \in (0,1), \\[1mm]
\lim_{x \downarrow 0} \big[g(x) - \wti g(0) (1-\b)x^{(1-\b)/2} \ln(1/x)\big]\\
\qquad\quad \times \big[(1-\b)^{-1}x^{(1-\b)/2}\big]^{-1}, & \hspace{-.2cm} \gamma =0.
\end{cases}
\end{align}
Choosing 
\begin{align}
u_1(0,x)&=u_{0,\a,\b,\gamma}(0, x)/u_{0,\a,\b,\gamma}(0, b),\hspace*{4.6cm} \g\in[0,1), \no \\
u_2(0,x)&=\begin{cases}
\hatt u_{0,\a,\b,\gamma}(0, x)-(1-\b)^2[(2+\a-\b) \gamma]^{-1} b^{-(2+\a-\b)\g}u_{0,\a,\b,\gamma}(0, x), \\
\hfill \g\in(0,1),\\[1mm]
\hatt u_{0,\a,\b,0}(0, x)-(1-\b)^2\ln(1/b)u_{0,\a,\b,0}(0, x),\hfill \g=0,
\end{cases} \no \\
&\hspace*{6cm} \a>-1,\; \b<1,\; x \in (0,b),
\end{align}
in \eqref{3.26}--\eqref{3.28} yields the Krein--von Neumann extension $T_{0,R_K,\a,\b,\g}$ of $T_{min,\a,\b,\g}$ in the form 
\begin{align}
& T_{0,R_K,\a,\b,\g} f = \tau_{\a,\b,\g,} f,    \\
& f \in \dom(T_{0,R_K,\a,\b,\g})=\bigg\{g\in\dom(T_{max,\a,\b,\g}) \, \bigg| \begin{pmatrix} \wti g(b) 
\\ {\wti g}^{\, \prime}(b) \end{pmatrix} = R_{K,\a,\b,\g} \begin{pmatrix}
\wti g(0) \\ {\wti g}^{\, \prime}(0) \end{pmatrix} \bigg\},     \no
\end{align}
where 
\begin{align}
R_{K,\a,\b,\g}&=\begin{cases}
b^{[\b-1-(2+\a-\b)\g]/2}\\
\quad \times\begin{pmatrix}  \dfrac{1-\b}{(2+\a-\b)\g}b^{1-\b} & \dfrac{1}{1-\b}b^{1-\b+(2+\a-\b)\g} \\
\dfrac{(1-\b)^2}{2(2+\a-\b)\g}-\dfrac{1-\b}{2} & \left[\dfrac{1}{2}+\dfrac{(2+\a-\b)\g}{2(1-\b)}\right] b^{(2+\a-\b)\g}
\end{pmatrix},     \\
\hfill \g\in(0,1), \\[1mm]
\begin{pmatrix}  (1-\b)\ln(1/b)b^{(1-\b)/2} & \dfrac{1}{1-\b}b^{(1-\b)/2} \\
\dfrac{(1-\b)^2\ln(1/b)-2(1-\b)}{2}b^{(\b-1)/2} & \dfrac{1}{2}b^{(\b-1)/2}
\end{pmatrix},
\hfill \g=0.
\end{cases}
\end{align}
One verifies that $det(R_{K,\a,\b,\g}) =1$. 
\end{example}
%%%%%%%

For the Krein extension of the standard Bessel operator on the half-line $(0,\infty)$ (i.e., $\alpha=\beta=0$, $a=0$, $b = \infty$) we also refer to \cite{BDG11} (see also \cite{AB16}, \cite{AB15}). 

\medskip

Next, we turn to a singular operator relevant to the phenomenon of acoustic black holes, following \cite{BHN21}.

%%%%%%%
\begin{example}[Acoustic Black Hole]\lb{ex4.2}
Let $a=0$ and $b\in(0,\infty)$ in \eqref{2.1} and consider
\begin{equation}\lb{4.12}
p(x)=p_0(x)x^{\alpha},\quad r(x)=r_0(x)x^{\beta},\quad q(x)=0,\quad x\in (0,b),
\end{equation}
where $\alpha,\beta\in \bbR$ are fixed and $p_0$ and $r_0$ are continuous real-valued functions on $(0,b)$ that satisfy 
for some $m,M\in (0,\infty)$, 
\begin{equation}\lb{4.13}
m \leq p_0(x)\leq M,\quad m\leq r_0(x)\leq M,\quad x\in (0,b). 
\end{equation}
Then
\begin{equation}\lb{4.14}
\tau_{p_0,r_0,\alpha,\beta} = -\frac{1}{r_0(x)x^{\beta}}\frac{d}{dx}p_0(x)x^{\alpha}\frac{d}{dx},\quad\alpha,\beta\in \bbR,\; x\in (0,b).
\end{equation}
Linearly independent solutions $y_j(0,\dott)$, $j\in \{1,2\}$, of the differential equation $\tau_{p_0,r_0,\alpha,\beta}y=0$ are given by
\begin{equation}\lb{4.15}
y_{1,p_0,r_0,\alpha,\beta}(0,x)=1,\quad 
y_{2,p_0,r_0,\alpha,\beta}(0,x)=\int_x^b\frac{dt}{p_0(t)t^{\alpha}},\quad \alpha,\beta\in \bbR,\; x\in (0,b),
\end{equation}
and they satisfy
\begin{equation}\lb{4.16}
W(y_{2,p_0,r_0,\alpha,\beta}(0,\dott),y_{1,p_0,r_0,\alpha,\beta}(0,\dott))=1.
\end{equation}
Explicit calculations reveal that
\begin{equation}\lb{4.17}
y_{1,p_0,r_0,\alpha,\beta}(0,\dott),y_{2,p_0,r_0,\alpha,\beta}(0,\dott)\in L^2\big((0,b);r_0(x)x^{\beta}dx\big),
\end{equation}
that is, $\tau_{p_0,r_0,\alpha,\beta}$ is in the limit circle case at $x=0$, if and only if $\beta>\max\{-1,2\alpha-3\}$.
In addition, $\tau_{p_0,r_0,\alpha,\beta}$ is regular at $x=b$.  
To avoid the scenario where $\tau_{p_0,r_0,\alpha,\beta}$ is in the limit point case at $x=0$, we thus assume that 
$\alpha$ and $\beta$ satisfy $\beta>\max\{-1,2\alpha-3\}$. To avoid that $\tau_{p_0,r_0,\alpha,\beta}$ is also regular 
at $x=0$ $($and hence regular on $(0,b)$$)$, we now also assume that $\alpha\geq 1$. $($There is no need to discuss the Krein--von Neumann extensions in the regular case as that can be found in \cite[Example~3.3]{CGNZ14}$)$. Altogether, this means we are assuming  
\begin{equation}
\alpha\geq 1 \, \text{ and } \, \beta> 2\alpha-3. 
\end{equation}
Under the assumption $\alpha\geq 1$, $y_1(0,\dott)$ and $y_2(0,\dott)$ are principal and nonprincipal solutions, respectively, of $\tau_{p_0,r_0,\alpha,\beta}u=0$ at $x=0$.  Thus, in accordance with Theorem \ref{t2.9}, one chooses
\begin{align}\lb{4.19}
\begin{split} 
u_{0,p_0,r_0,\alpha,\beta}(0,x)=1, \quad 
\widehat{u}_{0,p_0,r_0,\alpha,\beta}(0,x)=\int_x^b\frac{dt}{p_0(t)t^{\alpha}},&   \\
 \alpha\geq1,\; \beta>2\a-3,\;  x\in (0,b).& 
\end{split} 
\end{align}
The generalized boundary values for $g\in \dom(T_{max,p_0,r_0,\alpha,\beta})$ are then of the form
\begin{align}
\wti g(0)&=\lim_{x\downarrow 0}g(x)\bigg[\int_x^b\frac{dt}{p_0(t)t^{\alpha}}\bigg]^{-1},\lb{4.20}\\
\wti g^{\, \prime}(0)&=\lim_{x\downarrow 0}\bigg[g(x)-\wti g(0) \int_x^b\frac{dt}{p_0(t)t^{\alpha}}\bigg],\lb{4.21}\\
\wti g(b)&= \lim_{x\uparrow b}g(x) = g(b),\lb{4.22}\\
\wti g^{\, \prime}(b)&=\lim_{x\uparrow b}[pg'](x)=[pg'](b)=p_0(b)b^{\alpha}g'(b).\lb{4.23}
\end{align}
A basis for $\ker(T_{max,p_0,r_0,\alpha,\beta})$ which satisfies \eqref{3.26} is given by
\begin{align}\lb{4.24}
\begin{split} 
u_{1,p_0,r_0,\alpha,\beta} (0,x)=1,\quad 
u_{2,p_0,r_0,\alpha,\beta} (0,x)=\int_x^b\frac{dt}{p_0(t)t^{\alpha}},& \\
\alpha\geq1,\; \beta>2\a-3,\;   x\in (0,b),&
\end{split} 
\end{align}
and explicit calculations using \eqref{4.19}, \eqref{4.21}, and \eqref{4.23} reveal
\begin{align}\lb{4.25}
\begin{split} 
& \wti u_{1,p_0,r_0,\alpha,\beta}^{\, \prime}(0,0) = 1,\quad \wti u_{2,p_0,r_0,\alpha,\beta}^{\, \prime}(0,0)=0,   \\
& \wti u_{1,p_0,r_0,\alpha,\beta}^{\, \prime}(0,b) = 0,\quad \, \wti u_{2,p_0,r_0,\alpha,\beta}^{\, \prime} (0,b)=-1.
\end{split} 
\end{align}
Using \eqref{4.25} in \eqref{3.28} then yields 
\begin{equation}\lb{4.26}
R_{K,p_0,r_0,\alpha,\beta} =
\begin{pmatrix}
0 & 1\\
-1 & 0
\end{pmatrix}, \quad \alpha\geq1,\; \beta>2\a-3,  
\end{equation} 
and hence the Krein--von Neumann extension $T_{0,R_K,p_0,r_0,\alpha,\beta}$ of $T_{min,p_0,r_0,\alpha,\beta}$ is characterized by
\begin{align}
&T_{0,R_K,p_0,r_0,\alpha,\beta} f = \tau_{p_0,r_0,\alpha,\beta}f,\lb{4.27}\\
&\dom(T_{0,R_K,p_0,r_0,\alpha,\beta}) = \big\{g\in \dom(T_{max,p_0,r_0,\alpha,\beta})\,\big|\, 
g(b) = \wti g^{\, \prime}(0),\, g^{[1]}(b) = -\wti g(0)\big\}.      \no
\end{align}
\end{example}
%%%%%%% 

\medskip

Finally, we turn to the Jacobi operator referring to \cite{GLPS20} for a much more detailed analysis.

%%%%%%%
\begin{example}[Jacobi Operator] \lb{e4.3} 
Let \begin{align}
& a=-1, \quad b = 1,    \no \\
& p(x) = p_{\a,\b}(x) = (1-x)^{\a+1}(1+x)^{\b+1}, \quad q (x) = q_{\a,\b}(x) = 0,  \lb{4.29} \\ 
& r(x) = r_{\a,\b}(x) = (1-x)^\a (1+x)^\b, \quad x\in(-1,1),\quad \a,\b\in\bbR     \no 
\end{align}
$($see, e.g., \cite[Ch.~22]{AS72}, \cite{BFL20}, \cite[Sect.~23]{Ev05}, \cite{EKLWY07}, \cite{Fr20}, \cite{Gr74}, 
\cite{KKT18}, \cite{KMO05}, \cite[Ch.~18]{Ov20}, 
\cite[Ch.~IV]{Sz75}$)$ and 
consider the Jacobi differential expression 
\begin{align} \lb{4.28}
\begin{split} 
\tau_{\a,\b} = - (1-x)^{-\a} (1+x)^{-\b}(d/dx) \big((1-x)^{\a+1}(1+x)^{\b+1}\big) (d/dx),&     \\ 
x \in (-1,1), \; \a, \b \in \bbR.  
\end{split} 
\end{align}
To decide the limit point/limit circle classification of $\tau_{\a,\b}$ at the interval endpoints $\pm 1$, it suffices 
to note that if $y_1$ is a given solution of $\tau y = 0$, then a 2nd linearly independent solution $y_2$ of 
$\tau y = 0$ is obtained via the standard formula
\begin{equation}
y_2(x) = y_1(x) \int_c^x dx' \, p(x')^{-1} y_1(x')^{-2}, \quad c, x \in (a,b).    \lb{4.30} 
\end{equation}
Returning to the concrete Jacobi case at hand, one notices that 
\begin{align} 
\begin{split} 
& y_1(x) = 1, \quad x \in (-1,1),  \\
& y_2(x) = \int_0^x dx' \, (1 - x')^{-1-\a} (1+x')^{-1-\b},  \quad x \in (-1,1),      \lb{4.31} 
\end{split}
\end{align}
and hence 
\begin{align} 
&y_2(x)    \\
& \quad = \begin{cases}
2^{-1-\a} \b^{-1} (1+ x)^{-\b} [1+\Oh(1+ x)] + \Oh(1), & \a \in \bbR, \, \b \in \bbR\backslash\{0\}, \; \text{as $x \downarrow -1$}, \\
- 2^{-1-\a} \ln(1+ x) + \Oh(1), & \a \in \bbR, \, \b = 0, \; \text{as $x \downarrow -1$},  \\
2^{-1-\b} \a^{-1} (1 - x)^{-\a} [1+\Oh(1- x)] + \Oh(1), & \a \in \bbR\backslash\{0\}, \, \b \in \bbR, \; \text{as $x \uparrow +1$}, \\
- 2^{-1-\b} \ln(1 - x) + \Oh(1), & \a =0, \, \b \in \bbR, \; \text{as $x \uparrow +1$}.   
\end{cases}    \no 
\end{align} 
Thus, an application of Theorem \ref{t2.3}, Definition \ref{d2.4}, and Remark \ref{r2.7}\,$(ii)$ implies the classification,
\begin{equation}
\tau_{\a,\b} \, \text{ is } \begin{cases} \text{regular at $-1$ if and only if $\a \in \bbR$, $\b \in (-1,0)$,} \\
\text{in the limit circle case at $-1$ if and only if $\a \in \bbR$, $\b \in [0,1)$,} \\
\text{in the limit point case at $-1$ if and only if $\a \in \bbR$, $\b \in \bbR \backslash (-1,1)$,} \\
\text{regular at $+1$ if and only if $\a \in (-1,0)$, $\b \in \bbR$,} \\
\text{in the limit circle case at $+1$ if and only if $\a \in [0,1)$, $\b \in \bbR$,} \\
\text{in the limit point case at $+1$ if and only if $\a \in \bbR \backslash (-1,1)$, $\b \in \bbR$.}
\end{cases}     \lb{4.32}
\end{equation} 
The fact \eqref{4.31} naturally leads to principal and nonprincipal solutions 
$u_{\pm1,\a,\b}(0,x)$ and $\hatt u_{\pm1,\a,\b}(0,x)$ of $\tau_{\a,\b}y=0$ near $\pm1$ as follows$:$ 
\begin{align} 
\begin{split} 
u_{-1,\a,\b}(0,x)&=\begin{cases}
-2^{-\a-1}\b^{-1} (1+x)^{-\b}[1+\Oh(1+x)], & \b\in(-\infty,0),\\
1, & \b\in[0,\infty),
\end{cases}     \\
\hatt u_{-1,\a,\b}(0,x)&=\begin{cases}
1, & \b\in(-\infty,0),\\
-2^{-\a-1}\ln((1+x)/2), & \b=0,\\
2^{-\a-1} \b^{-1} (1+x)^{-\b} [1+\Oh(1+x)], & \b\in(0,\infty),
\end{cases}
\end{split} 
\quad \a\in\bbR,     \lb{4.35} 
\end{align}
and
\begin{align}
\begin{split} 
u_{+1,\a,\b}(0,x)&=\begin{cases}
2^{-\b-1} \a^{-1} (1-x)^{-\a} [1+\Oh(1-x)], & \a\in(-\infty,0),\\
1, & \a\in[0,\infty),
\end{cases}      \\
\hatt u_{+1,\a,\b}(0,x)&=\begin{cases}
1, & \a\in(-\infty,0),\\
2^{-\b-1}\ln((1-x)/2), & \a=0,\\
-2^{-\b-1} \a^{-1} (1-x)^{-\a} [1+\Oh(1-x)], & \a\in(0,\infty),  
\end{cases}
\end{split} 
\quad \b\in\bbR.     \lb{4.36} 
\end{align} 
Combining the fact \eqref{4.32} with Theorem \ref{t2.5}, the preminimal operator $\dot T_{min,\a,\b}$ corresponding to $\tau_{\a,\b}$ is essentially self-adjoint in $L^2((-1,1); r_{\a,\b} dx)$ if and only if 
$\a, \b \in \bbR \backslash (-1,1)$. Thus, boundary values for the maximal operator $T_{max,\a,\b}$ associated with 
$\tau_{\a,\b}$ at $-1$ exist if and only if 
$\a \in \bbR$, $\b \in (-1,1)$, and similarly, boundary values for $T_{max,\a,\b}$ at $+1$ exist if and only if 
$\a \in (-1,1)$, $\b \in \bbR$. 

Employing the principal and nonprincipal solutions \eqref{4.35}, \eqref{4.36} at $\pm 1$, according to 
\eqref{2.25}, \eqref{2.26}, generalized boundary values for $g\in\dom(T_{max,\a,\b})$ are of the form
\begin{align}
\begin{split} 
\wti g(-1)&=\begin{cases}
g(-1), & \b\in(-1,0),\\
-2^{\a+1}\lim_{x\downarrow-1}g(x)/\ln((1+x)/2), & \b=0,\\
\b 2^{\a+1}\lim_{x\downarrow-1}(1+x)^\b g(x), & \b\in(0,1),
\end{cases}    \\
{\wti g}^{\, \prime}(-1)&=\begin{cases}
g^{[1]}(-1), &\b\in(-1,0),\\
\lim_{x\downarrow-1}\big[g(x)+\wti g(-1)2^{-\a-1}\ln((1+x)/2)\big], & \b=0,\\
\lim_{x\downarrow-1}\big[g(x)-\wti g(-1)2^{-\a-1}\b^{-1}(1+x)^{-\b}\big], & \b\in(0,1),
\end{cases}
\end{split} 
\quad \a\in\bbR,      \\
\begin{split}
\wti g(1)&=\begin{cases}
g(1), & \a\in(-1,0),\\
2^{\b+1}\lim_{x\uparrow1}g(x)/\ln((1-x)/2), & \a=0,\\
-\a2^{\b+1}\lim_{x\uparrow1}(1-x)^\a g(x), & \a\in(0,1),
\end{cases}\\
{\wti g}^{\, \prime}(1)&=\begin{cases}
g^{[1]}(1), & \a\in(-1,0),\\
\lim_{x\uparrow1}\big[g(x)-\wti g(1)2^{-\b-1}\ln((1-x)/2)\big], & \a=0,\\
\lim_{x\uparrow1}\big[g(x)+\wti g(1)2^{-\b-1}\a^{-1}(1-x)^{-\a}\big], & \a\in(0,1),
\end{cases}
\end{split}
\quad \b\in\bbR.
\end{align}
For a detailed treatment of solutions of the Jacobi differential equation and the associated hypergeometric differential equations we refer to \cite[Appendix A]{GLPS20}.

To shorten the presentation of this example and hence avoid case $(i)$ in Theorem \ref{t3.5} where $\tau_{\a,\b}$ is in the limit point case at $-1$ or $+1$ and in the limit circle case at the opposite endpoint $($i.e., the cases where 
$n_{\pm}(T_{min,\a,\b}) = 1$$)$, we now assume that 
$\tau_{\a,\b}$ is in the limit circle case at $\pm 1$ $($i.e., $n_{\pm}(T_{min,\a,\b}) = 2$ as in case $(ii)$ of Theorem \ref{t3.5}$)$. Thus, we 
assume that 
\begin{equation}
\a,\b\in(-1,1)
\end{equation}
in the following. Next, we consider two linearly independent solutions of $\tau_{\a,\b}y=0$ near $x=-1$ given by
\begin{align}
y_{1,\a,\b,-1}(0,x)&=1,     \no \\
y_{2,\a,\b,-1}(0,x)&= 
(1+x)^{-\b}F(1+\a,-\b;1-\b;(1+x)/2), \quad \b\in(-1,1)\backslash\{0\},      \no \\
&\hspace*{5cm} \a\in(-1,1),\; x\in(-1,1). 
\end{align}
Furthermore, using the connection formulas found in \cite[Eq. 15.3.6, 15.3.10]{AS72} yields the behavior of $y_{2,\a,\b,-1}(0,x)$ near $x=1$,
\begin{align}
y_{2,\a,\b,-1}(0,x)&=\begin{cases}
(1+x)^{-\b}\dfrac{\Gamma(1-\b)\Gamma(-\a)}{\Gamma(-\a-\b)}F(1+\a,-\b;1+\a;(1-x)/2)\\
\ -(1+x)^{-\b}(1-x)^{-\a}2^\a \a^{-1}\b F(-\a-\b,1;1-\a;(1-x)/2),\\
\hfill \a\in(-1,1)\backslash\{0\},\\[1mm]
-(1+x)^{-\b}\b\displaystyle\sum_{n=0}^\infty \dfrac{(-\b)_n}{2^n(n!)^2}[\psi(n+1)-\psi(n-\b)\\
\quad -\ln((1-x)/2)](1-x)^n,\hfill \a=0,
\end{cases} \no \\
&\hspace*{4.5cm} \b \in(-1,1)\backslash\{0\},\; x\in(-1,1).
\end{align}
Here $F(\dott,\dott;\dott;\dott)$ denotes the hypergeometric function $($see, e.g., \cite[Ch.~15]{AS72}$)$, $\psi(\dott) = \Gamma'(\dott)/\Gamma(\dott)$ the Digamma function, $\gamma_{E} = - \psi(1) = 0.57721\dots$ represents Euler's constant, and 
\begin{equation}
(\zeta)_0 =1, \quad (\zeta)_n = \Gamma(\zeta + n)/\Gamma(\zeta), \; n \in \bbN, 
\quad \zeta \in \bbC \backslash (-\bbN_0), 
\end{equation}
abbreviates Pochhammer's symbol $($see, e.g., \cite[Ch.~6]{AS72}$)$. 

Similarly, we consider linearly independent solutions of $\tau_{\a,\b}y=0$ near $x=1$,
\begin{align}
y_{1,\a,\b,1}(0,x)&=1,    \no \\
y_{2,\a,\b,1}(0,x)&=
(1-x)^{-\a}F(1+\b,-\a;1-\a;(1-x)/2), \quad  \a\in(-1,1)\backslash\{0\},     \no \\
&\hspace*{5cm} \b\in(-1,1),\; x\in(-1,1),
\end{align}
noting one can show that 
\begin{align}
y_{2,\a,\b,1}(0,x)&=\begin{cases}
(1-x)^{-\a}\dfrac{\Gamma(1-\a)\Gamma(-\b)}{\Gamma(-\a-\b)}F(1+\b,-\a;1+\b;(1+x)/2)\\
\ -(1-x)^{-\a}(1+x)^{-\b}2^\b \b^{-1}\a F(-\a-\b,1;1-\b;(1+x)/2),\\
\hfill \b\in(-1,1)\backslash\{0\},\\[1mm]
-(1-x)^{-\a}\a\displaystyle\sum_{n=0}^\infty \dfrac{(-\a)_n}{2^n(n!)^2}[\psi(n+1)-\psi(n-\a)\\
\quad -\ln((1+x)/2)](1+x)^n,\hfill \b=0,
\end{cases} \no \\
&\hspace*{4.7cm} \a \in(-1,1)\backslash\{0\},\; x\in(-1,1).
\end{align}

For $\a,\b\in(-1,1)$, the following five cases are associated with a strictly positive minimal operator 
$T_{min,\a,\b}$ $($see, \cite{GLPS20}$)$ and we now provide the corresponding choices $u_1,\ u_2$ in 
\eqref{3.26}--\eqref{3.28} that yield $R_{K,\a,\b}$ and the Krein--von Neumann extension $T_{0,R_K,\a,\b}$ of $T_{min,\a,\b}$,
\begin{align}
& T_{0,R_K,\a,\b} f = \tau_{\a,\b} f,    \\
& f \in \dom(T_{0,R_K,\a,\b})=\bigg\{g\in\dom(T_{max,\a,\b}) \, \bigg| \begin{pmatrix} \wti g(1) 
\\ {\wti g}^{\, \prime}(1) \end{pmatrix} = R_{K,\a,\b} \begin{pmatrix}
\wti g(-1) \\ {\wti g}^{\, \prime}(-1) \end{pmatrix} \bigg\}. \no
\end{align}

\smallskip

%%%%%%%%%%%%
\noindent{\bf (I) The regular case $\boldsymbol{\a,\b\in(-1,0)}$:} Choosing 
\begin{align}
u_1(0,x)=2^\b\dfrac{\Gamma(-\a-\b)}{\Gamma(-\a)\Gamma(1-\b)}y_{2,\a,\b,-1}(0,x) , \quad
u_2(0,x)=1-u_1(0,x),
\end{align}
yields
\begin{align}
R_{K,\a,\b}&=
\begin{pmatrix}  1 & 2^{-\a-\b-1}\dfrac{\Gamma(-\a)\Gamma(-\b)}{\Gamma(-\a-\b)} \\
0 & 1
\end{pmatrix},\quad \a,\b\in(-1,0).
\end{align}
%%%%%%%%%%%%

\smallskip

%%%%%%%%%%%%
\noindent{\bf (II) The case $\boldsymbol{\a\in(-1,0),\; \b\in(0,1)}$:} Choosing 
\begin{align}
u_1(0,x)=1 , \quad
u_2(0,x)=\b^{-1}2^{-\a-1}y_{2,\a,\b,-1}(0,x)+2^{-\a-\b-1}\dfrac{\Gamma(-\a)\Gamma(-\b)}{\Gamma(-\a-\b)},
\end{align}
yields
\begin{align}
R_{K,\a,\b}&=
\begin{pmatrix}  -2^{-\a-\b-1}\dfrac{\Gamma(-\a)\Gamma(-\b)}{\Gamma(-\a-\b)} & 1 \\
-1 & 0
\end{pmatrix},\quad \a\in(-1,0),\; \b\in(0,1).
\end{align}
%%%%%%%%%%%%

\smallskip

%%%%%%%%%%%%
\noindent{\bf (III) The case $\boldsymbol{\a\in(0,1),\; \b\in(-1,0)}$:} Choosing 
\begin{align}
u_1(0,x)=\b^{-1}2^{-\a-1}y_{2,\a,\b,-1}(0,x), \quad u_2(0,x)=1 ,
\end{align}
yields
\begin{align}
R_{K,\a,\b}&=
\begin{pmatrix}  0 & -1 \\
1 & 2^{-\a-\b-1}\dfrac{\Gamma(-\a)\Gamma(-\b)}{\Gamma(-\a-\b)}
\end{pmatrix},\quad \a\in(0,1),\; \b\in(-1,0).
\end{align}
%%%%%%%%%%%%
$($In cases $\mathbf{(II)}$ and $\mathbf{(III)}$ we interpret $1/\Gamma(0) = 0$.$)$ 
\smallskip

%%%%%%%%%%%%
\noindent{\bf (IV) The case $\boldsymbol{\a=0,\; \b\in(-1,0)}$:} Choosing 
\begin{align}
u_1(0,x)=\b^{-1}2^{-1}y_{2,0,\b,-1}(0,x), \quad u_2(0,x)=1,
\end{align}
yields
\begin{align}
R_{K,0,\b}&=
\begin{pmatrix}  0 & -1 \\
1 & -2^{-\b-1}[\gamma_{E}+\psi(-\b)]
\end{pmatrix},\quad \a=0,\; \b\in(-1,0).
\end{align}
%%%%%%%%%%%%

\smallskip

%%%%%%%%%%%%
\noindent{\bf (V) The case $\boldsymbol{\a\in(-1,0),\; \b=0}$:} Choosing 
\begin{align}
u_1(0,x)=1, \quad u_2(0,x)=-\a^{-1} 2^{-1}y_{2,\a,0,1}(0,x),
\end{align}
yields
\begin{align}
R_{K,\a,0}&=
\begin{pmatrix}  2^{-\a-1}[\gamma_{E}+\psi(-\a)] & 1 \\
-1 & 0
\end{pmatrix},\quad \a\in(-1,0),\; \b=0.
\end{align}
%%%%%%%%%%%%

\noindent 
Obviously, $det(R_{K,\a,\b}) = 1$ in all five cases. 
\end{example}
%%%%%%%

%%%%%%%
\begin{remark} \lb{r4.4}
In the remaining four cases in Example \ref{e4.3}, given by all combinations of $\a=0,\ \b=0,\ \a\in(0,1)$, and $\b\in(0,1)$, one observes that Theorem \ref{t3.5} is not applicable as the underlying minimal operator, $T_{min, \a,\b}$, is not strictly positive. This is easily seen by considering the Jacobi polynomials and the boundary conditions they satisfy. The $n$th Jacobi polynomial is defined as $($see \cite[Eq. 18.5.7]{OLBC10}$)$
\begin{align}
\begin{split}
    P_n^{\a, \b}(x) := \dfrac{(\a+1)_n}{n!} F(-n, \, n + \a + \b+1; \,
     \a + 1; \, (1-x)/2),
     \\
     n\in\N_0,\ -\a\notin\N,\ -n-\a-\b-1 \notin \N,
\end{split}
\end{align}
and can be defined by continuity for all parameters $\a, \b \in \R$. We note that $P_n^{\a,\b}(x)$ is a polynomial of 
degree at most $n$, and has strictly smaller degree if and only if $-n-\a-\b \in \lbrace 1, \dots , n \rbrace$ 
$($cf.~\cite[p.~64]{Sz75}$)$. It satisfies the differential equation
\begin{align}
    \tau_{\a,\b} P_n^{\a, \b}(x) = \l^{\a,\b}_n  P_n^{\a, \b}(x),
\end{align}
where 
\begin{align}
    \l^{\a,\b}_n = n(n+1+\a+\b), \quad n \in \bbN_0.
\end{align}
One verifies that the Jacobi polynomials are solutions of the Jacobi operator eigenvalue equation 
$\tau_{\a,\b}y=\l^{\a,\b}_n y$ with Neumann boundary conditions in the regular case where $\a,\b\in(-1,0)$, 
and the Friedrichs boundary conditions in the present case under consideration where $\a, \b \in [0,1)$.

In particular, this implies that $0 \in\sigma(T_{F,\a,\b})$, $\a,\b \in [0,1)$, where $T_{F,\a,\b}$ denotes the Friedrichs 
extension of $T_{min, \a,\b}$, and hence $T_{min, \a,\b} \geq 0$ is nonnegative, but not strictly positive when 
$\a,\b \in [0,1)$.  
\hfill $\diamond$
\end{remark}
%%%%%%%

%%%%%%%%%%%%%%%%%%%%%%%%%%%%%%

%%%%%%%%%%%%%%%%%%%%%%%%%%%%%%
%%%%%%%%%%%%%%%%%%%%%%%%%%%%%%
%\appendix

\medskip

%%%%%%%%%%%%%%%%%%%%%%%%%%%%%%%%%%%%%
\noindent 
{\bf Acknowledgments.} We gratefully acknowledge discussions with Jussi Behrndt. 
We are indebted to Boris Belinskiy for kindly organizing the special session, 
``Modern Applied Analysis'' at the AMS Sectional Meeting at the University of Tennessee at Chattanooga,    
October 10--11, 2020, and for organizing the associated special issue in Applicable Analysis. 
%%%%%%%%%%%%%%%%%%%%%%%%%%%%%%%%%%%%%

%%%%%%%%%%%%%%%%%%%%%%%%%%%%%%%%
%%%%%%%%%%%%%%%%%%%%%%%%%%%%%%%%
 
\end{document}